\documentclass[11pt]{amsart}
\usepackage{amsthm}

\usepackage[all]{xy}
\usepackage{amssymb}
\usepackage{enumerate}
\usepackage{mathrsfs}
\usepackage{hyperref}
\usepackage{epsfig}
\usepackage{graphicx}
\usepackage{subfig}
\usepackage{float}
\usepackage{epigraph}

\numberwithin{equation}{section}

\newtheorem{thm}{Theorem}[section]

\newtheorem{lem}[thm]{Lemma}

\newtheorem{rem}{Remark}[section]

\newcommand{\eq}[1]{(\ref{#1})}

\renewcommand{\Re}{\operatorname{\rm Re}}
\renewcommand{\Im}{\operatorname{\rm Im}}

\newcommand{\beqast}{\begin{eqnarray*}}
\newcommand{\eqast}{\end{eqnarray*}}
\newcommand{\beqa}{\begin{eqnarray}}
\newcommand{\eqa}{\end{eqnarray}}

\newcommand{\bbe}{\begin{equation}}
\newcommand{\ee}{\end{equation}}

\newcommand{\sbr}{\smallbreak}

\newcommand{\bfo}{{\bf 1}}

\newcommand{\bE}{{\mathbb E}}

\newcommand{\bR}{{\mathbb R}}
\newcommand{\bC}{{\mathbb C}}
\newcommand{\bZ}{{\mathbb Z}}

\newcommand{\cC}{{\mathcal C}}

\newcommand{\eps}{\epsilon}

\newcommand{\al}{\alpha}
\newcommand{\be}{\beta}

  \newcommand{\sg}{\sigma}

\newcommand{\la}{\lambda}
\newcommand{\lp}{\lambda_+}
\newcommand{\lm}{\lambda_-}
\newcommand{\La}{\Lambda}

\newcommand{\om}{\omega}

\newcommand{\ze}{\zeta}

\newcommand{\ga}{\gamma}
\newcommand{\gap}{\gamma^+}
\newcommand{\gam}{\gamma^-}
\newcommand{\Ga}{\Gamma}

\newcommand{\Cp}{{C_+}}
\newcommand{\Cm}{{C_-}}

\newcommand{\cp}{{c_+}}
\newcommand{\cm}{{c_-}}

\begin{document}

\title
%[SINH-regular functions, distributions and processes I.]
%{SINH-regular functions, distributions and processes I. Infinitely divisible distributions and European options}
[Conformal accelerations method]
%: efficient evaluation of probability distributions,
%option pricing, calibration and Monte-Carlo simulations]
{Conformal accelerations method and efficient evaluation of stable distributions, revisited}

\author[
Svetlana Boyarchenko and
Sergei Levendorski\u{i}]
{
Svetlana Boyarchenko and
Sergei Levendorski\u{i}}

\thanks{In this revised version, the scheme for the polynomial acceleration is improved,
and wrong numbers in several lines of Table 1 are corrected.\\
\emph{S.B.:} Department of Economics, The
University of Texas at Austin, 1 University Station C3100, Austin,
TX 78712--0301, {\tt sboyarch@eco.utexas.edu} \\
\emph{S.L.:}
Calico Science Consulting. Austin, TX.
 Email address: {\tt
levendorskii@gmail.com}}

\begin{abstract}
We introduce new efficient integral representations and methods for evaluation of
pdfs, cpds and quantiles of stable distributions.
For wide regions in the parameter space, absolute errors of order $10^{-15}$ can be achieved
in 0.005-0.1 msec (Matlab implementation), even when the index of the distribution is small or close to 1. 
 For the calculation of quantiles in  wide regions in the tails using
the Newton or bisection method, it suffices to precompute several hundred values of 
the characteristic exponent
at points of an appropriate grid (conformal principal components) and use these values in formulas for cpdf and pdf,
which require a fairly small number of elementary operations.
The methods of the paper are applicable to other classes of integrals, especially highly oscillatory ones, and are typically faster than the popular methods.

\end{abstract}
\maketitle


\section{Introduction}\label{intro}
Stable L\'evy processes and densities appear in various fields of natural sciences, engineering and finance. 
 See, e.g., 
\cite{Chandrasekhar,Chavanis09,Mand1,Mand2,SamorodnitskyTaqqu94,Nolan97,Nolan98,Nolan03,NikiasShao95,SignalProc10,AmentONeil18} and the bibliographies therein.
Recall that
the L\'evy density of a one-dimensional stable L\'evy process $X$ of index $\al\in (0,2)$ is  of the form
\beqa\label{SLLdens}
\nu(dy)&=&|y|^{-\al-1}(\cp \bfo_{(0,+\infty)}(y)
%\\\nonumber&&
+\cm \bfo_{(-\infty,0)}(y))dy,
\eqa
where $c_\pm\ge 0$ and $\cp+\cm>0$. The L\'evy-Khintchine formula
\bbe\label{LKh}
\psi(\xi)=\int_\bR (1-e^{iy\xi}+iy\xi \bfo_{(-1,1)}(y))\nu(dy)-ib\xi,
\ee
for the characteristic exponent $\psi$ of a pure jump L\'evy process, definable from $\bE\left[e^{iX_t\xi}\right]=e^{-t\psi(\xi)}, \xi\in\bR$, allows one to calculate the characteristic exponent
$\psi_{st}$ of a stable L\'evy process. 
If $\al\neq 1$, then $\psi_{st}(\xi)=-i\mu\xi+\psi^0_{st}(\al, \Cp, \xi)$, where $\mu$ can be expressed in terms of $\al$, $c_\pm$ and $b$,
\bbe\label{SLnuneq10}
\psi^0_{st}(\xi)=\Cp |\xi|^\al\bfo_{(0,+\infty)}(\xi)+ \Cm|\xi|^\al\bfo_{(-\infty,0)}(\xi),
\ee
$\Cm=\overline{\Cp}$, $\bar z$ denotes the complex conjugate to $z$,
and 
$\Cp=\Cp(\al,\cp,\cm)$ is given by
\beqa\label{Cp}
\Cp&=&-\cp\Ga(-\al)e^{-i\pi\al/2}-\cm\Ga(-\al)e^{i\pi\al/2}.
\eqa
 With $\al=2$ and $C_\pm=\sg^2/2$, we have the characteristic exponent of the Brownian motion (BM).
 If $\al=1$, then the formula for $\psi^0_{st}$ is more involved:
  \bbe\label{SLnu1}
\psi^0_{st}(\xi)=\sg|\xi|(1+i(2\be/\pi)\,\mathrm{sign}\,\xi\ln|\xi|),
\ee
where $\sg=(\cp+\cm)\pi/2$, $\be=(\cp-\cm)/(\cp+\cm)$. This is a version of Zolotarev's  parametrizations \cite{Zolotarev}
for stable processes of index 1; for processes of index $\al\neq 1$, the corresponding
parametrization
 used in \cite{SamorodnitskyTaqqu94} is
 \beqa\label{SLnuneq1}
\psi^0_{st}(\xi)&=&
%\psi^0_{st}(\al,\sg,\beta,\xi)\\&:=&
\sg^\al(1-i\be\tan(\al\pi/2)\,\mathrm{sign}\,\xi)|\xi|^\al,
\eqa
where $\sg>0$ and $\be\in [-1,1]$ are related to $c_\pm$ as follows:
\beqa\label{cpcmsgbe}
\sg^\al&=&-\Ga(-\al)\cos(\al\pi/2)(\cp+\cm),\ \be=\frac{\cp-\cm}{\cp+\cm}.
\eqa
For our approach to the calculation of pdf $p$, cpdf $F$, quantiles,
and other quantities related to $X_1$ (we normalize $t$ to 1), it is important 
to regard $p$, $p'$ and $F$ as functions of $x'=x-\mu$ rather than of the pair $(x,\mu)$:
 \beqa\label{pdfSL}
 p(x')&=&\frac{1}{2\pi}\int_\bR e^{-ix'\xi-\psi^0_{st}(\xi)}d\xi,\\\label{pdfSL1/2}
 &=&\frac{1}{\pi}\Re\int_{\bR_+}e^{-ix'\xi-\psi^0_{st}(\xi)}d\xi,\
 \\\label{derpdfSL}
 p'(x')&=&\frac{1}{\pi}\Re\int_{\bR_+} (-i\xi)e^{-ix'\xi-\psi^0_{st}(\xi)}d\xi,\\\label{cpdfSL}
 F(x')&=&\frac{1}{2}+\mathrm{v.p.}\frac{1}{2\pi}\int_\bR\frac{e^{-ix'\xi-\psi^0_{st}(\xi)}}{-i\xi}d\xi.
 \eqa
 (Here v.p. denotes the Cauchy principal value.) Furthermore, some of the constructions and proofs in the paper are simpler
  in terms of $\Cp$, $|\Cp|$ and $\varphi_0=\mathrm{arg}\,\Cp=-\arctan(\be\tan(\al\pi/2))$.
 Parameterizations \eq{SLnu1} and \eq{SLnuneq1} are convenient in the case $\al\neq 1$, when  the rescaling
$\xi\to a\xi$, where $a>0$, is needed to change the ratio $x'/C_+$.
As we will explain in the main body of the text, the rescaling allows one to decrease the number of terms in the quadrature procedures
that we will construct. If $\al=1$, the rescaling is not so useful.

For wide regions in the parameter space, the integrands in \eq{pdfSL}-\eq{cpdfSL}
are highly oscillatory, and standard numerical quadratures face serious difficulties. 
Zolotarev \cite{Zolotarev} used the method of stationary phase to reduce the calculation of pdf to evaluation
of an integral of a positive function, which has exactly one point of maximum (see \cite{Nolan97} for a variation of this approach).
Unfortunately, in many cases, the spike of the integrand is very high,
and popular numerical algorithms can miss the spike and underestimate the integral. To avoid this problem,
one  needs to find the point of maximum numerically and compute the integral as a sum of two integrals. The procedure becomes
computationally expensive. In addition, the new integrand is given by a complicated expression, hence
even the computational cost of evaluation at one point is non-negligible.
See \cite{Zolotarev,SamorodnitskyTaqqu94,Nolan97,Nolan98,AmentONeil18} and the bibliographies therein for details.
Finally, this approach is very difficult to generalize for the case of mixtures of stable distributions, which arise in applications to signal processing  \cite{SignalProc10},
because there may be several points of local extremum.

In  a recent paper \cite{AmentONeil18},
the authors analyze difficulties of numerical realization of \eq{pdfSL}-\eq{cpdfSL} for several classes of popular quadrature methods and asymptotic expansions,
 and suggest a small collection of quadrature rules and asymptotic expansions which allow one
to evaluate the integrals \eq{pdfSL}-\eq{cpdfSL}  for large regions in the parameter space.
The class of quadrature schemes in \cite{AmentONeil18} is an extension of the classical Gaussian quadrature
schemes, which require the precalculation of nodes and weights with sufficiently high precision; the calculations are rather involved.
The authors  state that, in the asymmetric case, the set of quadrature rules in \cite{AmentONeil18} is efficient if $\al\in [0.5,0.9]\cup [1.1,2.0]$;
in the symmetric case, the asymptotic formulas are efficient for small $\al$, and the quadratures for $\al\in [0.5,2]$. The numerical experiments reported
in \cite{AmentONeil18} produce the pdf for $\al\in[0.5,0.9]\cup [1.1,2.0]$
 and the cpdf for $\al\in [1.1, 2]$ with absolute errors of order $10^{-14}$. For the cpdf in the
 range $\al\in [0.5, 0.9]$, absolute errors of order $10^{-8}$ are documented.

In the present paper, we introduce families of changes of variables which lead to integrals that can be calculated quickly using the simplified 
trapezoid rule. This approach does not rely on  very accurate evaluation of auxiliary quantities such as the point of maximum
in \cite{Zolotarev,Nolan97}, allows for fast oscillation of the integrand, and, unlike the stationary phase method,
is fairly flexible as far as the choice of an approximately optimal change of variables is concerned. In fact, a couple of the most efficient integral representations which 
we derive (see \eq{pdfSL2yLe1p2}  and \eq{cpdf_conic_nuLe14}) involve factors 
of the form $\sin(a\sin(b e^{\al y}))$, where $a,b>0$, $\al\in (0,1)$;
as $y\to +\infty$, the integrands oscillate widely. 
Methods developed in the present paper allow one to calculate $p$, $F$ and derivatives of $p$ w.r.t. $x$ and $\al, \sg, \be$
accurately and quickly for $\al$ close to 1 (e.g., $\al=0.998$) or 0 (e.g., $\al=0.1$), and/or large $|x'|$ using fairly simple schemes which require 1-2 hundred 
terms in the simplified trapezoid rule
 to satisfy error tolerance of order $10^{-15}$ (for cpdf, $10^{-12}$); in some regions, several dozen of terms suffice.
 In 90-95\% cases, the CPU time is in the range of 5-100 microseconds, and, with the exception of a very small region $\{\al\in (0,1/3), 0<\be x'<<1\}$,
 not more than several msec. suffice. Typically, the universal procedures designed in the paper allow one
 to choose the parameters of the schemes automatically, and double-check the accuracy of
 the results if desired.

\begin{figure}
  \includegraphics[width=1\textwidth]{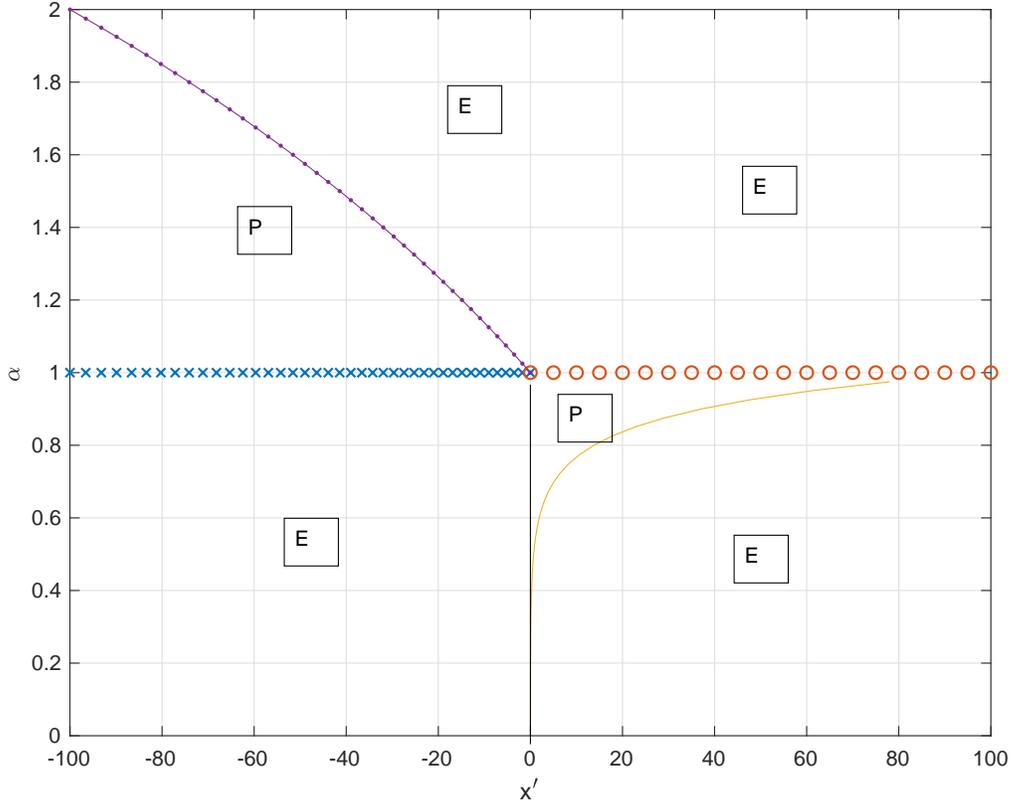}
\caption{Stylized regions, for $\be>0$. E: where the exponential acceleration and simplified conic
trapezoid rule are efficient. P: the polynomial acceleration is more efficient. 
Crosses:  where the sub-polynomial acceleration is more efficient.
E is the union of 1) the vertical solid segment
and the region to the left of it; 2) the region to right of the solid curve; 3) line of circles and 
the region above this line and dash-dotted line. }
\label{fig:1}       % Give a unique label
\end{figure}

The main ingredient of the first method is a conformal change deformation of the contour of integration  in \eq{pdfSL}, the corresponding change of variables 
of the form
\begin{equation}\label{sinhbasic}
\xi=\chi_{\om_1, \om; b}(y)=i\om_1+b\sinh (i\om+y),
\end{equation}
where $i=\sqrt{-1}$, $\om_1, \om \in \bR, b>0$ ({\em sinh-acceleration}), and the simplified trapezoid rule in the $y$-coordinate.
The sinh-acceleration is possible if the characteristic exponent admits analytic continuation to a union of a strip containing the real line
or adjacent to the real line and a cone (see \cite{SINHregular} for applications of the sinh-acceleration to several problems
in finance).
 In the case of completely asymmetric L\'evy processes
(either $\cp=0$ or $\cm=0$), the characteristic exponent admits analytic continuation
to the complex plane with a cut along one of the imaginary half-axes, hence sinh-acceleration
is possible.
 The parameters of the change of variables and the grid for the simplified
  trapezoid rule needed to satisfy the desired error tolerance can be calculated quickly and easily; the summation in the simplified trapezoid
  rule admits a straightforward  vectorization.
  Furthermore, one can use two changes of variables to double-check the accuracy of the calculations: the probability that
  two sums for two different changes of variables agree up to, e.g., $10^{-13}$ purely by chance is negligible unless each sum
  has an error of the order of $10^{-13}$.
  In applications to evaluation of complicated integrals arising in computational finance,
  the efficiency of this approach as compared to popular quadratures is demonstrated in \cite{iFT,paraHeston,MarcoDiscBarr,one-sidedCDS,pitfalls,BarrStIR,UltraFast,HestonCalibMarcoMeRisk,SINHregular}.
  Note that in op.cit. (with exception of \cite{SINHregular}), a family of conformal changes of variables
  (called {\em fractional-parabolic})
  \begin{equation}\label{fractparabasic}
\xi=\chi^\pm_{\om;a; b}(\eta)= i\om\pm ib(1\mp i\eta)^a,
\end{equation}
where $\om\in\bR, a>1, b>0$, was used. The sinh-acceleration and 
fractional-parabolic change of variables are applicable iff the characteristic exponent admits analytic continuation
to the union of a strip and cone containing $\bR$ or adjacent to $\bR$.
If the parameters are chosen correctly, then the oscillating factor becomes fast decaying one, and the rate of the decay
of the integrand at the infinity increases exponentially in the case of the sinh-acceleration and polynomially in the case of the fractional-parabolic change
of variables. 
Stable L\'evy processes of index $\al\in (0,2)$ with $c_\pm>0$ do not admit analytic continuation
to a strip around or adjacent to the real axis,
hence, the sinh-acceleration can be applied only after an approximation of $\psi^0_{st}$ 
by functions that do enjoy this property.  In the first version \cite{NewFamilies} of the present paper, 
we use slight modifications of the characteristic exponents of
%two classes of L\'evy processes with exponentially decaying tails:
KoBoL processes \cite{KoBoL,NG-MBS}
\begin{eqnarray*}\label{kbl2}
\psi^0(\xi)&=&-c_-\Gamma(-\nu)(\lp+i\xi)^\nu
%\\\nonumber &&\hskip1.1cm
-c_+\Ga(-\nu)(-\lm-i\xi)^\nu,
\end{eqnarray*}
where $\nu\in (0,2), \nu\neq 1$, $c_\pm\ge 0$ and $c_++c_->0$ (if $\nu=1$,
the expression is different). 
We calculated the approximate values for several small
values of $\la=\la_+=-\la_-$ using sinh-acceleration, and
 applied
 Richardson's extrapolation. 
Our numerical experiments show
that the  method based on Richardson extrapolation
is less efficient than the  methods explained below.

We modify the sinh-acceleration, fractional parabolic acceleration, and the third type of the conformal changes
of variables
\bbe\label{subpolynomial}
\xi=\eta\ln^m(1+b\eta^2)
\ee 
where $m>1, b>0$ (this is similar to the {\em hyperbolic} change of variables introduced in \cite{iFT},
in addition to the fractional parabolic change of variables). Here, the acceleration 
of the rate of convergence is sub-polynomial but,  as we will explain in the paper, in some cases, the change of variables \eq{subpolynomial}
is preferable. Similarly, 
as examples in \cite{SINHregular} demonstrate, typically, 
\eq{sinhbasic} is more efficient than \eq{fractparabasic}, but this is not always the case.

We use the following one-sided counterparts of \eq{sinhbasic}, \eq{fractparabasic}, \eq{subpolynomial}. 
 \smallbreak
 \noindent
(E)  
We make the exponential change of the variable $\xi=e^{i\om+y}$ in the integral \eq{pdfSL1/2}
 \bbe\label{pdfSL2}
p(x')=\frac{1}{\pi}\Re\int_0^{+\infty} e^{i\om+y-ix'\xi(y)-\psi^0_{st}(\xi(y))}dy,
\ee
where the choice of $\om$ is determined by a cone where the integrand decays quickly,
 and apply the simplified trapezoid rule in the $y$-coordinate ({\em simplified conic trapezoid rule})
 \bbe\label{conic_simpl}
 p(x')=\frac{\ze}{\pi}\Re\left(e^{i\om}\sum_{j=-N_-}^{N_+} e^{j\ze-ix'e^{i\om+j\ze}-\psi^0_{st}(e^{i\om+j\ze})}\right).
 \ee
  In many cases, the  integrand in the $y$-coordinate decreases more slowly as $y\to-\infty$ than as $y\to +\infty$,
  hence, we have to choose $N_->> N_+$. To decrease $N_-$,
  we calculate several terms of
 the asymptotic expansion of the truncated part of the infinite sum in the neighborhood  of $y=-\infty$, and add the result to the
 sum in the simplified trapezoid rule.
 For the scheme to
 be efficient, we need to derive sufficiently simple and approximately optimal prescriptions for the choice of the cone of analyticity
 of the initial integral, and for the choices of $\om$ and $\ze, N$. These choices  rely on the analysis of the cone of analyticity
 of the initial integrand, and bounds for the discretization error (in terms of the Hardy norm) and truncation error.
 
  \smallbreak
 \noindent
(P) If $x'<0$, we make the change of variables
\bbe\label{pol1/2m}
\xi=i(1-(1+ib\eta^m)^a),
\ee
where $a\in (1,2)$, $m\ge 1$ and $b>0$; if $x'>0$, we use 
\bbe\label{pol1/2p}
\xi=i(-1+(1-ib\eta^m)^a).
\ee 
This counterpart of
\eq{fractparabasic} is of a more general form than \eq{fractparabasic}; the latter  can be generalized
in the similar vein. After that, we apply an additional change of variables
\bbe\label{yPchange}
\eta=0.5(y+(y^2+1)^{1/2}),
\ee
and apply the simplified trapezoid rule. The generalization \eq{pol1/2m}-\eq{pol1/2p} is crucial in this paper because 
if $m=1$, then $N_-$ in the simplified
trapezoid rule is too large. Contrary to (E), the terms of the asymptotic expansion are not given by simple explicit analytical expressions.
They are expressed in terms of Riemann zeta function.  

 \smallbreak
 \noindent
(SubP) We make the change of variables \eq{subpolynomial}, then the change of variables 
\bbe\label{yPchange2}
\eta=0.5(y+(y^2+D^2)^{1/2}),
\ee
choose an appropriate line of intregration $\{\Im y =\mp D/2\}$ (the choice of sign depends on the sign of the product
$x'\be$), and, finally, apply the simplified trapezoid rule. 
The terms of the asymptotic expansion are expressed in terms of Riemann zeta function and its derivatives.

\smallbreak
In the case of cpdf, it is useful or even necessary to make  simple preliminary transformations and after that apply the conformal changes of variables.

\smallbreak
Thus, we have 3 two-sided conformal acceleration methods, which can be applied to completely asymmetric L\'evy processes
(and many classes of L\'evy processes with exponentially decaying tails used in finance), and 3 one-sided versions, which
increase the rate of decay of the integrand exponentially (E), polynomially (P), and sub-polynomially (SubP). If the cone of analyticity, 
where the integrand decays fast at infinity, is not too narrow, then (E) is strictly better than (P), and (P) is better than (SubP).
However, if the cone is too narrow, then the mesh size in (E) is much smaller than in (P) and (SubP), and, for a given error tolerance,
the number of terms in (P) or (SubP) can be significantly smaller. See Fig.1 for a stylized description of the regions
where this or that method is advantageous.

\smallbreak

 The rest of the paper is organized as follows.
 In Sect.~ \ref{conicpdfnuneq1}-\ref{conicpdfnu1}, we apply the exponential change of variables and simplified trapezoid rule to pdf.
In Sect.~\ref{coniccpdfnuneq1}-\ref{coniccpdfnu1}, the same scheme is applied to the cpdf.
 In Sect.~\ref{der_and_quantiles}, we explain how to modify the method of the paper for calculation
of the derivatives of pdf and cpdf w.r.t.  $x'$ and parameters of the distribution.
In Sect.~\ref{algquantile}, we outline the procedure for evaluation of quantiles in the tails, and applications to the Monte Carlosimulations.
An application of   the sinh-acceleration in the case of completely asymmetric L\'evy processes
 is outlined in Sect.~\ref{one-sided}. Polynomial and sub-polynomial accelerations are explained in Sect. \ref{s:Polynom}
and \ref{s:SubP}.
Numerical examples are in Sect.~\ref{numer}.  Sect.~\ref{concl} summarizes the results of the paper.

 \section{Exponential acceleration. PDF, $\al\neq 1$}\label{conicpdfnuneq1}
 If $x'=0$, the pdf $p(x')$ can be calculated explicitly:
%\subsection{Proof of \eq{pdfstable}}\label{proof:pdfstable}
% \beqast
%p_{st}(0)&=&\frac{1}{\pi}\Re\int_0^{+\infty} e^{-\Cp \xi^\nu}d\xi
%=\frac{1}{\pi\nu}\Re C_+^{-1/\nu}\int_0^{+\infty}y^{1/\nu-1}e^{-y}dy.%=\frac{(1/\nu)\Ga(1/\nu)}{\pi}\Re C_+^{-1/\nu}.
%\eqast
%The integral is easily calculated, and we obtain
\bbe\label{pdfstable}
p(0)
=\frac{\Ga(1/\al+1)}{\pi}\Re C_+^{-1/\al}=\frac{\Ga(1/\al+1)}{\pi |\Cp|^{1/\al}}\cos\frac{\varphi_0}{\al}.
\ee
Hence, in this section, we assume that $x'\neq 0$.

 \subsection{Choice of the cone of analyticity and $\om$}\label{choice_cone}
 We need to choose an open cone $\cC_{\gam_0,\gap_0}=\{\mathrm{arg}\, \xi\in (\gam_0,\gap_0)\}$ around $\bR_+$ or adjacent to $\bR_+$  s.t.
$e^{-ix'\xi-\psi^0_{st}(\xi)}$, the integrand in \eq{pdfSL}, admits analytic continuation to $\cC_{\gam_0,\gap_0}$ and decays as $\xi\to \infty$ along any ray
in $e^{i\varphi}\bR_+\subset \cC_{\gam_0,\gap_0}$;
if  $\gam_0<-\pi$ or $\gap_0>\pi$, then
the cone $\cC_{\gam_0,\gap_0}$ is a subset of an appropriate Riemann surface.
% If $\cC_{\gam_0,\gap_0}$
%then $\psi^0_{st}(\xi)$ admits analytic continuation to $\cC_{\gam_0,\gap_0}$, and $-ix'\xi$ is analytic in $\bC$.
Since $x'\neq 0$ and $\al\neq 1$, an equivalent condition for the decay at infinity is: for any $\varphi\in (\gam_0,\gap_0)$, as $\rho\to+\infty$,
\begin{enumerate}[(i)]
\item if $\al\in (1,2]$,
$\Re\psi^0_{st}(\rho e^{i\varphi})\to+\infty$;
\item
if $\al\in (0,1)$,
$\Re(ie^{i\varphi}x'\rho)\to +\infty.$
\end{enumerate}
As $\rho\to+\infty$,
\bbe\label{asimpsi0nuneq1}
\psi^0_{st}(\rho e^{i\varphi})\sim  |C_+| e^{i(\varphi_0+\varphi\al)}\rho^\al.
\ee
Hence, (i) holds with the choice $\ga^\pm_0=(\pm\pi/2-\varphi_0)/\al$.
If $x'<0$ (resp., $x'>0$), (ii) holds with $\gam_0=0, \gap_0=\pi$ (resp., $\gam_0=-\pi, \gap_0=0$).
 However,
simplified recommendations   that ensure that $\Re\psi^0_{st}(\rho e^{i\varphi})\to+\infty$ if $\al\in (1,2)$, 
and $\Re(ie^{i\varphi}x'\rho)\to +\infty$ if $\al\in (0,1)$,
for $\varphi\in (\gam_0,\gap_0)$, but not that  both $\Re\psi^0_{st}(\rho e^{i\varphi})\to+\infty$ and $\Re(ie^{i\varphi}x'\rho)\to +\infty$ for $\varphi\in (\gam_0,\gap_0)$,
can imply very large Hardy norms (hence, very small mesh sizes), and very large truncation
parameters. The number of terms $N$ in the resulting simplified trapezoid rule 
becomes very large. If  $\Cp$ and $x'$ are of approximately the same order of magnitude, the following choice
is approximately optimal in the sense that no other choice can significantly decrease $N$ needed to satisfy
a given error tolerance:
\begin{enumerate}[a.]
\item
if $\al\in [1.05,2]$ or $x'=0$, then $\ga^\pm_0=(\pm\pi/2-\varphi_0)/\al$;
\item
if $\al\in (0,0.95]$, then  $\gam_0=0, \gap_0=\pi$ if $x'<0$, and $\gam_0=-\pi, \gap_0=0$ if $x'>0$.
\end{enumerate}
If $\Cp$ and $x'$ are not of the same order of magnitude but
$\al$ is not too close to 1, we can use the rescaling $\xi\mapsto a\xi$ in
the initial formula for the stable density (with $a>>1$ and $a<<1$ in the first case and the second case, respectively),
and reduce to the case of $\Cp$ and $x'$ of the same order of magnitude.

If $\al$ is close to 1, the rescaling is inefficient, and
%and either $\nu\in (1,2]$ and $|C_+|<<|x'|$ or $\nu\in (0,1)$ and $|C_+|>>|x'|$,
it is simpler to use $\gam_0, \gap_0$ such that
both (i) and (ii) hold.
Explicitly,
\begin{enumerate}[c.]
\item
if $x'>0$, $\gam_0=-\min\{\pi, (\pi/2+\varphi_0)/\al\}$, $\gap_0=0$;
\end{enumerate}
\begin{enumerate}[d.]
\item
if $x'<0$, $\gam_0=0$, $\gap_0=\min\{\pi, (\pi/2-\varphi_0)/\al\}$.
\end{enumerate}
 Prescriptions c.- d. can be used for all $\al\neq 1$, and, in the majority of cases,
do not lead to  significant increases of the number of terms in the simplified trapezoid rule.
The ``exceptionally bad" regions, where simple prescriptions c.-d. lead to unnecessary narrow
strips of analyticity in the $y$-coordinate, hence,  unnecessarily large numbers of terms in
the simplified trapezoid rule, are the ones where $x'$ and $\varphi_0$ are of the opposite sign,
$|\varphi_0|$ is close to $\pi/2$, and the ratio $\Cp/x'$ is very small (resp., large)
if $\al>1$ (resp., $\al<1$). Equivalently, the number of terms is very large if
either 
\sbr
\noindent
(1) $\al>1$ is close to 1, $\beta$ is close to $-\mathrm{sign}\,x'$, and $\Cp/x'$ is very small, or 
\sbr
\noindent
(2) $\al<1$ is close to
1, $\beta$ is close to $\mathrm{sign}\,x'$, and $\Cp/|x'|$ is very large.
\sbr
In these cases, the following simple recommendation allows one to reduce the number of terms in the simplified trapezoid rule:
\sbr
\noindent
$e$. if $\al>1$ is close to 1, find the maximal subinterval $(\gam_0,\gap_0)\subset (-\pi/2-\varphi_0,\pi/2-\varphi_0)$ 
s.t. $\forall \ga\in (\gam_0,\gap_0)$
\bbe\label{condgauniv0}
  \Re(ix'e^{i\ga}+\Cp e^{i\al\ga})>0,
  \ee
  equivalently,
  \bbe\label{condgauniv}
  -x'\sin\ga+|\Cp|\cos(\varphi_0+\al\ga)>0;
  \ee
\sbr
\noindent
$f_+$. if $\al<1$ is close to 1, $x'>0$,  and $\be>0$, set $\gap_0=0$, and find the minimal $\gam_0$ s.t. \eq{condgauniv}
holds;
\sbr
\noindent
$f_-$.
if $\al<1$ is close to 1, $x'<0$,  and $\be<0$, set $\gam_0=0$, and find the maximal $\gap_0$ s.t. \eq{condgauniv}
holds.

\begin{rem}\label{complicated_choice}{\rm One can derive more efficient albeit more involved
recommendations analyzing the dependence of the Hardy norm and truncation parameter below
on the choice of $\ga^\pm_0$.}
\end{rem}
 
\subsection{Error bound of the infinite trapezoid rule and choice of $\ze$}\label{choice_ze}
We  set
$\om=(\gam_0+\gap_0)/2$, $d_0=(\gap_0-\gam_0)/2$, $d=k_dd_0$, where $k_d<1$ (e.g., in the range $[0.8,0.95]$).
With this choice, the integrand in \eq{pdfSL2}, denote it $f(y)$, 
admits analytic continuation
to the strip $S_{(-d,d)}=\{y\in \bC\ |\ \Im y\in (-d,d)\}$ around the real line and decays sufficiently fast as $y\to\infty$ remaining
in the strip so that
$
\lim_{A\to \pm\infty}\int_{-d}^d |f(ia+A)|da=0,$
and the Hardy norm
%\bbe\label{Hnorm0}
\[
H(f,d)=\lim_{a\downarrow -d}\int_\bR|f(ia+ y)|dy+\lim_{a\uparrow d}\int_\bR|f(ia+y)|dy
\]
%\ee
is finite.
Fix $\ze>0$ and
construct the grid $\{y_j=j\ze, j\in \bZ\}$. The  error
of the infinite trapezoid rule
\begin{equation}\label{pdfLevytrap0}
p_{st}(x')=\ze\sum_{j\in \bZ} f(y_j)
\end{equation}
 admits an upper bound via $H(f,d)\exp[-2\pi d/\ze]/(1-\exp[-2\pi d/\ze])$  (see Theorem 3.2.1 in \cite{stenger-book} and Appendix in \cite{paraHeston} for a simple proof).
 In some cases, the Hardy norm of the integrand as a function on a maximal strip of analyticity
is infinite. In such cases, in order that the universal bound for the discretization error be applicable, one has to apply the bound
to functions on
a narrower strip of analyticity; this explains our choice  $d<d_0$. We use
an approximate upper bound $H(f,d)\le H_++H_-$, where
\bbe\label{boundHpm0}
H_\pm=\frac{1}{\pi }\int_0^{+\infty}e^{x'\sin(\om\pm d)\rho-c_\infty(\om\pm d)\rho^\al}d\rho,
\ee
and $c_\infty(\varphi)=|C_+|\cos(\varphi_0+\varphi\al)$. To derive approximate bounds for $H_\pm$,
one can use any simple quadrature or design simple general prescriptions. 
We consider separately cases $\al\in (1,2]$, $\al\in (0,1)$.
\sbr
\noindent
(1) If $\al\in (1,2]$, then $c_\infty(\om\pm d)>0$. We change the variable $c_\infty(\om\pm d)\rho^\al=u$:
\bbe\label{boundHpm012}
H_\pm=\frac{1}{\pi \al c_\infty(\om\pm d)^{1/\al}}\int_0^{+\infty}u^{1/\al-1}e^{B_\pm u^{1/\al}-u}du,
\ee
where $B_\pm=x'\sin(\om\pm d)c_\infty(\om\pm d)^{-1/\al}$. If $B_\pm\le 0$, the integral on the RHS of \eq{boundHpm012}
is bounded by $\Ga(1/\al)<1$; if $B_\pm\le 1/2$, then by $2^{1/\al}\Ga(1/\al)+\al e^{1/2}$.
\sbr
\noindent
(2) If $\al\in (0,1)$, then $-x'\sin(\om\pm d)>0$. We write
\bbe\label{boundHpm001}
H_\pm=\frac{1}{\pi(-x'\sin(\om\pm d)) }\int_0^{+\infty}e^{-u+B_\pm u^\al}d\rho,
\ee
where $B_\pm=-c_\infty(\om\pm d)(-x'\sin(\om\pm d))^{-\al}$. If $B_\pm\le 0$, the integral on the RHS
of \eq{boundHpm001} is bounded by 1; if $B_\pm\le 1/2$, then by 2.

\sbr
\noindent
(3) In both cases (1) and (2), to obtain a bound for $H_\pm$ when $B_\pm$ in the region $(1/2,+\infty)$, one can store the values of the integrals for several values of $B_\pm$ and use
a simple interpolation procedure. It is evident that the integral can be very large if $\al$ is close to 1 or $B_\pm$ is very large, this is why, for $\al$ close to 1,
we
recommend to choose $\ga^\pm_0$ so that $B_\pm\le 0$, and, in the cases of large positive $B_\pm$, make the preliminary rescaling so that
the new $B_\pm$ are not too large.

\sbr
\noindent
(4) In Cases e.-f., the choice of 
$\ga^\pm_0$ is made so that \eq{condgauniv0}, equivalently, \eq{condgauniv}
holds, and the simplified bound is via
\beqa\label{boundHpm00ef}
H_\pm&=&\frac{1}{\pi}\int_0^{+\infty}e^{-(-x'\sin(\om\pm d)+c_\infty(\om\pm d))\rho}d\rho%\\\label{boundHpm00ef}&=&
=1/(\pi(-x'\sin(\om\pm d)+c_\infty(\om\pm d))).
\eqa
 We set $\ze=2\pi d/\ln(10(H_++H_-)/\eps)$, where $\eps>0$ is the error tolerance for the
discretization error.

\subsection{Choice of $N_+$}\label{choice_N_+}
The infinite sum $\ze\sum_{j>N_+}f(j\ze)$ can be approximated by the integral
%\begin{equation}\label{truncerror0}
$Err_{tr}=\int_\La^{+\infty} |f(y)|dy.$
%\end{equation}
If $\La$ is large, then we use  an approximation
\bbe\label{La1tJ000}
Err_{tr}= \frac{1}{\pi}\int_{\La_1}^{+\infty}e^{x'\sin(\om)\rho-c_\infty(\om)\rho^\al}d\rho,
\ee
where $\La_1=e^\La$, and $\La=N\ze$. To derive approximate bounds for $\La_1$
(we do not need accurate ones; relative errors of the order of 1000\% are admissible
and can be taken into account introducing an additional factor 10), one needs to consider separately several cases.
\sbr
\noindent
(1) If $A:=-x'\sin\om>0$, and $B:=c_\infty(\om)>0$, we find $\La_1$ from
the equation
\[
e^{-A\La_1-B\La^\al_1}\le \eps_1:=-\eps\pi x'\sin\om,
\]
where $\eps$ is the error tolerance for the truncation error, equivalently, as the positive solution of the equation
\[
F(u)=Au+Bu^\al-C=0,
\]
where $C=\max\{1,\ln(1/\eps_1)\}>0$. The
 solution can be easily  found using Newton's method. 
 Since there is no need to know $\La_1$ with high precision, we define  $\La_1=\min\{C/A, (C/B)^{1/\al}\}$.
\sbr
\noindent
(2)
If $B=c_\infty(\om)<0$, then $\al\in (0,1)$ and $-x'\sin\om>0$.  Assuming that $\al$ is not very close to 1,
we can find an approximation to $\La_1$ as follows. First, we find $\La_{11}$ such that, for $\rho\ge \La_{11}$,
\[
x'\sin(\om)\rho-c_\infty(\om)\rho^\al\le x'\sin(\om)\rho/2,
\]
next, find $\La_{12}=\frac{2\ln(1/\eps_1)/}{-x'\sin\om}$ from the condition
\[
\int_{\La_{12}}^{+\infty}e^{x'\sin(\om)\rho/2}d\rho\le \pi\eps,\]
and then set $\La_1=\max\{\La_{11},\La_{12}\}$.

\sbr
\noindent
(3) If $B=c_\infty(\om)>0$ but $-x'\sin\om<0$, then $\al\in (1,2]$. In \eq{La1tJ000}, we
change
the variable $c_\infty(\om)\rho^\al=u$:
%\bbe\label{LatJ003}
\[
Err_{tr}=\frac{1}{\pi \nu c_\infty(\om\pm d)^{1/\al}}\int_{\La_2}^{+\infty}u^{1/\nu-1}e^{B_\pm u^{1/\al}-u}du,
\]
%\ee
where $B_\pm=x'\sin(\om\pm d)c_\infty(\om\pm d)^{-1/\al}$ and $\La_2= c_\infty(\om)\La_1^\al$. 
Assuming that $\al$ is not very close to 1,
we can find an approximation to $\La_2$ as follows. First, we find $\La_{21}$ such that, for $u\ge \La_{21}$,
$
B_\pm u^{1/\al}-u\le -u/2$, next, find $\La_{22}$ from the condition
\[
\int_{\La_{22}}^{+\infty}u^{1/\al-1}e^{-u/2}du=\eps_2:=\eps\pi\al c_\infty(\om\pm d)^{1/\al}.
\]
We can use $\La_{22}=2\ln(1/\eps_2)$ as an upper bound for the solution; a more accurate bound
can be obtained solving the equation
\[
u/2+(1-1/\al)\ln u-\ln(2/\eps_2)=0.
\]
Then we
set $\La_2=\max\{\La_{21},\La_{22}\}$, $\La_1=(\La_2/c_\infty(\om))^{1/\al}$.

\sbr
\noindent
(4)
In Cases e.-f. of Sect. \ref{choice_cone}, the choice of
$\ga^\pm_0$ is made so that \eq{condgauniv0}, equivalently, \eq{condgauniv}
holds, and the simplified equation for $\La_1$ is
\bbe\label{La1ef}
\int_{\La_1}^{+\infty}e^{-(-x'\sin(\om\pm d)+c_\infty(\om\pm d))\rho}d\rho=\pi\eps.
\ee
We find $\La_1$, and set $\La=\ln(\La_1)$, $N_+=\mathrm{ceil}(\La/\ze)$.

\subsection{Choice of $N_-$}\label{choice_N_-}
A reasonably accurate approximation to $N_-$ can be found as follows:
 $\La_-=-\ln(\eps\pi)$, $N_-=\mathrm{ceil}(\La_-/\ze)$. We see that if $\al\in (1,2]$ or $\al\in (0,1)$ and $|x'|$ is large, then
for a small error tolerance, e.g., $10^{-15}$, $N_+$ can be several dozen and smaller 
whereas $N_-$ is 1-2 hundred. To decrease $\La_-$, hence, $N_-$, we use
\begin{eqnarray}\label{difp2}
p(x')&=&p(0)+\frac{1}{\pi}\Re\left(e^{i\om}\int_\bR
%{-\infty}^{+\infty} 
(e^{-ix'e^{i\om+y}}-1)e^{y-\Cp e^{\al(i\om+y)}}dy\right).
\end{eqnarray}
 The term $p(0)$ is calculated explicitly \eq{pdfstable}, and the second term on the RHS
 is calculated applying the simplified conic trapezoid rule
  to the integral on the RHS:
 \bbe\label{difpimp}
p(x')=p(0)+\frac{\ze }{\pi}\Re \left(e^{i\om}\sum_{j=-N_-}^{N_+}f(y_j)\right),
 \ee
 where $N_-=0.5\mathrm{ceil}\, (\ln(|x'|/(\eps\pi))/\ze)$, and
 \[
 f(y)=e^y[e^{-ix'e^{i\om+y}-C_+e^{\al(i\om+y)}}-e^{-\Cp e^{\al(i\om+y)}}].
 \] 
The key parameters $\ga^\pm_0$ should be chosen so that both $\Re(ix' e^{i\varphi}+\Cp e^{i\al\varphi})$ and $\Re(\Cp e^{i\al\varphi})$ are positive for any
 $\varphi\in (\gam_0,\gap_0)\}$ (the choice c.-d. in Sect.~\ref{choice_cone}).
 Then we may use the same approximate recommendations for $\ze$ and $N_+$ as above,
 but decrease $N_-$. 
 We start with \eq{difp2} and  note that 
 \bbe\label{fpdfnuneq1}
 f(y)=-ix'e^{i\om}e^{2y}+C(y)R_2(y), 
 \ee
 where $C(y)\to 1$ as $y\to -\infty$, and $R_2(y)$ admits the upper bound
 \bbe\label{boundR2pdfnuneq1}
 |R_2(y)|\le (x')^2 e^{3y}/2+|x'||\Cp|e^{(2+\al)y}.%+|\Cp|^2e^{1+2\al}.
 \ee
 The representation \eq{fpdfnuneq1} and similar expansions in the next sections allows one to
 add correction terms to the  simplified conic trapezoid rule. These terms are of the form $a_{rs}S_{rs}(\ze,N_-)$, where $a_{rs}\in\bC$ are
independent of $\ze$ and $N_-$ (some of $a_{rs}$ depend on $x'$), and
\bbe\label{Srs}
S_{rs}(\ze,N_-):=\sum_{j=-\infty}^{-N_--1}e^{j(r+s\al)\ze}=\frac{e^{-(r+s\al)(N_-+1)\ze}}{1-e^{-(r+s\al)\ze}}.
\ee
 If \eq{difp2} is used,
 \beqa\label{pdfSL2y3}
 p(x')&=&\frac{\Ga(1/\al+1)}{\pi |C_+|^{1/\al}}\cos\frac{\varphi_0}{\al} +\frac{\ze }{\pi}\Re \left(e^{i\om}\sum_{j=-N_-}^{N_+}f(j\ze)
%\right.\\\nonumber&&\left.
-ix'e^{2i\om}S_{20}(\ze, N_-)\right).
 \eqa
 
  \subsection{The case of $\al< 1$ and large $|x'|$}\label{conicpdfnuLeq1}
 In this case, 
  it is advantageous to transform \eq{pdfSL2} as follows. If $x'>0$, then $\om\in (-\pi,0)$, and if $x'<0$, then $\om\in (0,\pi)$.
  In both cases,
  %\bbe\label{expomLe1}
  \[
  \frac{e^{i\om}}{\pi}\int_{-\infty}^{+\infty} e^{y-ix'e^{i\om+y}}dy=\frac{e^{i\om}}{\pi}\int_0^{+\infty}e^{-ix'e^{i\om}\rho}d\rho=\frac{1}{\pi ix'}.
  \]
 % \ee
  Since $\Re(1/(\pi ix'))=0$, we may rewrite \eq{pdfSL2} as follows
  \begin{eqnarray}\label{pdfSL2yLe1}
 p(x')=\frac{1}{\pi}\Re \left(e^{i\om}\int_\bR f(y)dy\right),
\end{eqnarray}
where $f(y)=e^{y-ix'e^{i\om+y}-\Cp e^{\al(i\om+y)}}- e^{y-ix'e^{i\om+y}}$. 
We use the recommendations in Subsection \ref{choice_cone}
to choose $\ga^\pm_0, \om, d_0$. 
  The choice of $N_+$ is modified in the evident manner, and $N_-$ depends on the order
  of the asymptotic expansion.  We use
the Taylor expansion of the exponent of order 1 or 2.\footnote{The reader can easily derive
expansions of higher order; in our numerical experiments, these expansions did not bring
sizable advantages, and, naturally, if $|x'|$ is very large, even the expansion of order 3 should be avoided.}
 \sbr\noindent
{\em Ord 1.}  We define $\La_{11}=\ln(2|x'\Cp|/(\eps\pi))/(2+\al)$, 
$\La_{12}=\ln(|\Cp|^2/(\eps\pi))/(1+2\al)$, $\La_1=\max\{\La_{11},\La_{12}\}$, $N_-=\mathrm{ceil}(\La_1/\ze)$,
calculate $S_{11}(\ze,N_-)$,
%\bbe\label{le1S1}
%S_1(N_-)=\frac{e^{-(1+\al)(N_-+1)\ze}}{1-e^{-(1+\al)\ze}},
%\ee
and then 
 \beqast%\label{pdfSL2yLe11}
 p(x')=\frac{\ze}{\pi}\Re \left(e^{i\om}\sum_{j=-N_-}^{N_+}
 f(j\ze)-\Cp e^{i\om(1+\al)}S_{11}(\ze,N_-)\right).
 \eqast
 
\sbr\noindent
{\em Ord 2.}  We define $\La_{21}=\ln(3(x')^2|\Cp|/(2\eps\pi))/(3+\al)$, 
$\La_{22}=\ln(3|x'\Cp^2|/(2\eps\pi))/(2+2\al)$, \\
$\La_{23}=\ln(|\Cp^3|/(6\eps\pi))/(1+3\al)$, 
$\La_1=\max_j\La_{2j}$, $N_-=\mathrm{ceil}(\La_1/\ze)$,
calculate $S_{11}(\ze,N_-)$, $S_{21}(\ze,N_-)$  and $S_{12}(\ze,N_-)$,
%\beqa\label{le1S21}
%S_{21}(N_-)&=&\frac{e^{-(2+\al)(N_-+1)\ze}}{1-e^{-(2+\al)\ze}}\\\label{le1S22}
%S_{22}(N_-)&=&\frac{e^{-(2\al+1)(N_-+1)\ze}}{1-e^{-(2\al+1)\ze}}
%\eqa
and then 
 \beqast%\label{pdfSL2yLe12}
 p(x')&=&\frac{\ze}{\pi}\Re e^{i\om}\left(\sum_{j=-N_-}^{N_+}
 f(j\ze)-\Cp e^{i\om(1+\al)}S_{11}(\ze,N_-)
 \right.\\\nonumber
 &&+ \left. ix'\Cp e^{i(2+\al)\om} S_{21}(\ze,N_-)
 %\\\nonumber
%&& \left.
\frac{\Cp^2}{2}e^{i\om(1+2\al)}S_{12}(\ze,N_-)\right).
 \eqast

\subsection{The case $\al< 1$. Further simplifications}\label{NuLe1further}
If $|x'|+|\Cp|\cos(\varphi_0-(\mathrm{sign}\,x' )\al\pi/2)$ is positive and not very small, then \eq{pdfSL2yLe1} can be
  simplified using\\ $\om=-\mathrm{sign}\, x' \pi/2$;
  the number of elementary operations needed to calculate the individual terms in the simplified trapezoid rule
  decreases. Since there is no  sizable loss in the width of the strip of analyticity, the number of terms
  remains approximately the same, and the total CPU time decreases.
 \sbr 
If $x'<0$, then $\om=\pi/2$, and $d_0\le\pi/2$ needed to choose $\ze$ is found from the condition
%\bbe\label{gapmLe1m}
\[
-x'\sin\ga+|\Cp|\cos(\varphi_0+\al\ga)>0, \ga\in (\frac{\pi}{2}-d_0,\frac{\pi}{2}+d_0).
\]
%\ee
The formula for the pdf becomes
\beqa\label{pdfSL2yLe1m2}
p(x')
%\label{pdfSL2yLe1m}
 &=&-\frac{1}{\pi}\Im \int_\bR\left( e^{y+x'e^{y}-|\Cp| e^{i(\varphi_0+\frac{\al\pi}{2}}e^{\al y}}
%\right.\right.\\\nonumber
%&&\left. \left. 
- e^{y+x'e^{y}}\right)dy\\\nonumber
 &=&\frac{1}{\pi}\int_\bR
 e^{y+x'e^{y}-|\Cp| \cos(\varphi_0+\frac{\al\pi}{2})e^{\al y}}
 \sin\left(|\Cp| \sin(\varphi_0+\frac{\al\pi}{2})e^{\al y}\right)dy.
 \eqa
If $x'>0$, then $\om=-\pi/2$, and $d_0\le\pi/2$ needed to choose $\ze$ is found
from the condition
%\bbe\label{gapmLe1p}
\[
-x'\sin\ga+|\Cp|\cos(\varphi_0-\al\ga)>0,  \quad \forall\ 
\ga\in (-\frac{\pi}{2}-d_0,-\frac{\pi}{2}+d_0).
\]
%\ee
The formula for the pdf becomes
\beqa\label{pdfSL2yLe1p2}
 p(x)
%\label{pdfSL2yLe1p}
&=&-\frac{1}{\pi}\Im \int_\bR\left( e^{y-x'e^{y}-|\Cp| e^{i(\varphi_0-\frac{\al\pi}{2})}e^{\al y}}
- e^{y-x'e^{y}}\right)dy\\\nonumber
 &=&\frac{1}{\pi}\int_\bR
 e^{y-x'e^{y}-|\Cp| \cos(\varphi_0-\frac{\al\pi}{2})e^{\al y}}%\\\label{pdfSL2yLe1p2}&&\hskip2cm\times
\sin\left(|\Cp| \sin(\varphi_0-\frac{\al\pi}{2})e^{\al y}\right)dy.
 \eqa
We leave to the reader the application of the simplified trapezoid rule and
straightforward calculations of the real parts of the coefficients 
in the improved formula for the pdf. %\eq{pdfSL2yLe13}.

\begin{rem}{\rm The integral on the RHS of \eq{pdfSL2yLe1m2} and \eq{pdfSL2yLe1p2} are highly
oscillatory but if $-x'$ is relatively large w.r.t. $|\Cp|\cos(\varphi_0+\al\pi/2)$ (resp., $x'$ is relatively large
w.r.t. $|\Cp|\cos(\varphi_0-\al\pi/2)$), then the integrands in the formulas for the Hardy norm and
the truncation error decay very fast at infinity, hence, a small number of terms in the simplified 
trapezoid rule suffices to satisfy a small error tolerance. If $|x'|$ is insufficiently large,
then the formulas in this subsection are inefficient.
}
\end{rem}

\section{Exponential acceleration. PDF, $\al= 1$}\label{conicpdfnu1}
In the symmetric case $\cp=\cm$,
 the pdf of the stable L\'evy process of index 1 can be easily calculated, hence, we consider
the asymmetric case $\cp\neq \cm$.
Since we cannot calculate $p(0)$ in the closed form, we replace the straightforward analog of \eq{difp2} with
 %\bbe\label{pdfconicnuneq1}
 \beqast
p(x')&=&\Re \frac{1}{\pi C_0}
%\\&&
+\frac{1}{\pi}\Re\int_0^{+\infty}\left(e^{-ix'\xi-\psi^0(\xi)}-e^{-C_0\xi}\right)d\xi,
\eqast
%\ee
where $C_0$ is in the open right half-plane, and apply the simplified conic trapezoid rule to the integral on the RHS.
 To make the impact of the additional term under the integral sign on the choice of the parameters of the simplified conic trapezoid rule as
small as possible, we choose  $C_0$ in the form $c_0e^{-i\om}$, where $c_0>0$ is moderately large.
The result is the formula
\bbe\label{pdfSL2ynu10}
 p(x')=\frac{\cos\om}{c_0\pi}+\frac{\ze }{\pi}\Re \left(e^{i\om}\sum_{j=-N_-}^{N_+}f(j\ze)\right),
 \ee
 where  \[
 f(y) =
 \exp[e^{i\om+2y}(-i(x'+\frac{2\sg\be}{\pi}y)-\sg(1-\frac{2\sg\be}{\pi})]
-\exp(y-c_0 e^y). \] 
Let $\varphi\in (-\pi/2,\pi/2)$, and $\xi=\rho e^{i\varphi}$. As $\rho\to+\infty$,
%\beqa\nonumber
%\psi^0(\xi)&=&i\rho e^{i\varphi}\left[\cp \ln\left(\rho e^{i(-\pi/2+\varphi)}\right)-\cm\ln\left(\rho e^{i(\pi/2+\varphi)}\right)\right]
%+o(|\xi|)\\\nonumber
%&=&\rho e^{i\varphi}\left[i(\cp-\cm)\ln\rho+(\cp+\cm)(\pi/2) -(\cp-\cm)\varphi\right]+o(|\xi|).
%\eqa
%Simplifying,
%\bbe\label{asimpsi0nu1}
\[
\psi^0_{st}(\xi)=\rho \sg e^{i\varphi}\left[1+i(2\be/\pi)\ln\rho-(2\be/\pi)\varphi\right]+o(\rho).
\]
%\ee
The leading term of asymptotics is $i e^{i\varphi}(2\be/\pi)\rho\ln\rho$, and we need to
deform the contour to the region where
\[
c_\infty(\varphi):=\Re (i e^{i\varphi})\frac{2\sg\be}{\pi}=-\frac{2\sg\be}{\pi}\sin\varphi>0.
\]
Hence, if $\be>0$ (resp., $\be<0$),  we deform the contour downward, and use $\gap_0=0$ (resp., upward, and use $\gam_0=0$).
Formally, in the first case, we may use $\gam_0=-\pi$, and in the second case, $\gap_0=\pi$. However, this is a reasonable choice only
if either $x'\be\ge 0$, or $x'\be<0$ but $-x'/(\sg\be)$ is not large. If $x'$ and $\be$ are of the opposite sign,
and $ |x'|/(\sg\be)$ is  large, we have
\[
e^{-ix'\rho e^{i\varphi}-\psi^0_{st}(\rho e^{i\varphi})}\sim  e^{-(ix'+\sg)e^{i\varphi}\rho}
\]
in the region of large $\rho$ and $x'$ s.t. $|x'|>> \sg|\be|2/\pi \ln \rho$. Hence, in this case, we need to choose $\gam_0<\gap_0$ so that,
for all $\varphi\in (\gam_0, \gap_0)$,
$
\Re (ix'+\sg)e^{i\varphi}>0$, equivalently,  $-x'\sin\varphi+\sg\cos\varphi>0$.
 In order to take these subtleties into account, we give the
following recommendations (in some cases, they are inefficient, but fairly safe).
\subsection
{Recommendation for the choice of $\gam_0,\gap_0, d, \om$}\label{choice_cone_1}
We consider the following 4 cases.
\sbr\noindent
$(+,+)$ 
if $\be>0$ and $x'\ge -\frac{\sg\be}{2}$, then $\gam_0=-\pi, \gap_0=0$;
\sbr
\noindent
$(-,-) $
if $\be<0$ and $x'\le -\frac{\sg\be}{2}$, then $\gam_0=0, \gap_0=\pi$;
\sbr
\noindent
$(+,-) $
if $\be>0$ and $x'\le -\frac{\sg\be}{2}$, then $\gam_0=\arctan(\frac{\sg}{x'}), \gap_0=0$;
\sbr
\noindent
$(-,+)$ 
if $\be<0$ and $x'\ge -\frac{\sg\be}{2}$, then $\gam_0=0, \gap_0=\arctan(\frac{\sg}{x'})$.

\sbr
\noindent
Choices of $\om$ and $d$ are determined by $\gam_0, \gap_0$ as in the case $\al\neq 1$.

\subsection{Approximate bounds for the Hardy norm and truncation error, and choice of $\ze$ and $N_+$}\label{choice_ze_1}
For $\varphi\in (\gam_0,\gap_0)$, we use the following simple (and fairly accurate if $\rho$ is large) bound
%\bbe\label{Hboundnu1}
\[
\left|e^{-ix'\rho e^{i\varphi}-\psi^0_{st}(\rho e^{i\varphi})}\right|\le e^{((x'+2\be\sg/\pi)\sin \varphi-\sg(1-2\be\om/\pi)\cos\varphi)\rho}
\]
%\ee
to derive an approximation
for the Hardy norm $H\le H_++H_-$, where
%\bbe\label{HpmSinhnu1}
\beqast
H_\pm&=&\frac{1}{\pi}\left(\sg(1-\frac{2\be\om}{\pi})\cos(\om\pm d)-(x'+\frac{2\be\om}{\pi})\sin (\om\pm d)\right)^{-1}.
\eqast
%\ee
We set $\ze=2\pi d/\ln(10 H/\eps)$.
Using the bound
\beqast
&&
\frac{1}{\pi}\int_{\La_1}^{+\infty}\left|e^{-ix'\rho e^{i\om}-\psi^0(1,\sg,\be,\rho e^{i\om})}\right|d\rho
%\\&\le &
\le \frac{ e^{-\La_1(\sg(1-2\be\om/\pi)\cos\om-(x'+2\be\sg/\pi)\sin \om)}}{\pi[\sg(1-2\be\om/\pi)\cos\om-(x'+2\be\sg/\pi)\sin \om]},
\eqast
we derive an approximation to the truncation parameter $\La$. First, we define
\begin{eqnarray}\label{La1nu1h}
B&=&\sg(1-2\be\om/\pi)\cos\om-(x'+2\be\sg/\pi)\sin \om, \\\nonumber
 \eps_1&=&\eps\pi B,\
\La_1= \ln (1/\eps_1)/B,
\end{eqnarray}
and then set $
\La_{1,+}=\max\{\La_1,-\frac{\ln(\eps\pi c_0\cos d)}{c_0\cos d}\}, \La_+=\ln \La_{1,+},\ N_+=\mathrm{ceil}\,(\La_+/\ze).
$

\subsection{Choice of $N_-$}\label{choice_N_-}
We can use  $N_-=\mathrm{ceil}\,(\La_-/\ze)$, where
$\La_-=\ln(1/\eps)$, which requires an unnecessary large $N_-$. To decrease $N_-$, we use the asymptotic expansion
\beqa\label{fpdfnu1} %\nonumber%\label{deffpdfnu1}
f(y)&=&e^y\left(e^{-ix'e^{i\om +y}-\sg e^{i\om +y}(1+i(2\be/\pi)(y+i\om))}-e^ {-c_0e^{y}}\right)\\\nonumber
 &=&-((ix'+\sg(1-2\be\om/\pi))e^{i\om}-c_0)e^{2y}+(2i \sg \be/\pi)e^{i\om}e^{2y}(-y) +C(y)e^{3y}R_2(y),
 \eqa
 where $C(y)\to 1$ as $y\to-\infty$, and $R_2(y)$ admits the upper bound
 \beqa\nonumber
 |R_2(y)|&\le& \frac{|x'|^2+\sg^2|1-2\be\om/\pi|^2+c_0^2}{2}+\frac{2\sg^2\be^2}{\pi^2}|y|^2
\\\label{boundR2pdfnu1}&&+\frac{(|x'|^2+\sg^2|1-2\be\om/\pi|^2)^{1/2}2 \sg |\be|}{\pi}|y| \eqa
 We write the RHS of \eq{boundR2pdfnu1} as $a_0+a_1|y|+a_2y^2$, define $\La_{-,1}=\ln(3a_0/(\eps\pi))/3$,
 find $\La_{-,j}$ as
 an approximate solution to $e^{-3\La}\La^j=\eps\pi/(3a_j)$, $j=1,2$, set $\La_-=\max\La_{-,j}$, $N_-=\mathrm{ceil}\,(N_-/\ze)$.
 Then we calculate $S_{20}(\ze,N_-)$ (see
  \eq{Srs}), and $S^1_2(\ze,N_-)$, where $S^1_n(\ze,N_-)$ is given by
 \[
 S^1_n(\ze,N_-)=\sum_{j=-\infty}^{-N_--1}e^{nj\ze}(-j).
 \]
 To derive a formula for $S_1(\ze,N_-)$, we
differentiate the equality
 $
 \sum_{j=N_-+1}^{+\infty} q^j=\frac{q^{N_-+1}}{1-q}, |q|<1,
 $
w.r.t. $q$. The result is
\beqast
 \sum_{j=N_-+1}^{+\infty}j q^{j-1}&=&\frac{(N_-+1)q^{N_-}(1-q)+q^{N_-+1}}{(1-q)^2}
 =\frac{q^{N_-}((N_-+1)-N_-q)}{(1-q)^2}.
 \eqast
 Substituting $q=e^{-n\ze}$, we obtain
 \bbe\label{S1}
S^1_n(\ze,N_-)=\frac{e^{-n(N_-+1)\ze}((N_-+1)-N_-e^{-n\ze})}{(1-e^{-n\ze})^2}.
\ee
Finally, we obtain the formula similar to \eq{pdfSL2ynu10},
with two additional terms:
\beqa\label{pdfSL2ynu1}
 p(x')&=&\frac{\cos\om}{c_0\pi}
+\frac{\ze }{\pi}\Re \left(e^{i\om}\sum_{j=-N_-}^{N_+}f(j\ze)-A_0S_{20}(\ze,N_-) +A_1S^1_2(\ze, N_-)\right),
\eqa
% \\\label{pdfSL2ynu12}
% && \hskip0.5cm \left.\left.-((ix'+\sg(1-2\be\om/\pi)e^{i\om}-c_0)S_0(N_-)+\frac{2i \sg \be\ze}{\pi} \right)\right).
 %\eqa
 where $A_0=(ix'+\sg(1-2\be\om/\pi))e^{2i\om}-c_0e^{i\om},  A_1=2i e^{2i\om}\sg \be\ze/\pi$.
 
 \section{Exponential acceleration. CPDF, $\al\neq 1$}\label{coniccpdfnuneq1}
\subsection{Main formulas}\label{main_cpdf_nuneq1}
Taking into account that $
\mathrm{v.p.}\int_\bR e^{-|\xi|}/(-i\xi)d\xi=0,
$
 we obtain
 \bbe\label{cpdf2}
F(x')=\frac{1}{2}+J_{-1}(\al,\Cp)+F_{1}(x'),
\ee
where 
\beqa\label{J-1}
J_{-1}(\al,\Cp)&=&\frac{1}{2\pi}\int_\bR\frac{e^{-\psi^0_{st}(\xi)}-e^{-|\xi|}}{-i\xi}d\xi,\\\nonumber
F_{1}(x')&=&\frac{1}{2\pi}\int_\bR \frac{e^{-ix'\xi}-1}{-i\xi}e^{-\psi^0_{st}(\xi)}d\xi\\\nonumber
&=&\frac{1}{\pi}\Re \int_0^{+\infty} \frac{e^{-ix'\xi}-1}{-i\xi}e^{-\psi^0_{st}(\xi)}d\xi\\\label{JRN2}
&=&-\frac{1}{\pi}\Im \int_0^{+\infty} \frac{e^{-ix'\xi}-1}{\xi}e^{-\Cp\xi^\al(\xi)}d\xi.
\eqa
We calculate $J_{-1}(\al,\Cp)$ as follows.
Let $a=\Re \Cp, b=\Im \Cp$. If $b=0$, then $J_{-1}(\al,\Cp)=0$. We calculate the derivative
\beqast
\partial_{b} J_{-1}(\al,\Cp)&=&\frac{1}{\pi }\Re\int_0^{+\infty}\xi^{\al-1}e^{-(a+ib) \xi^\al}d\xi\\
&=&
\frac{1}{\pi\al}\Re (a+ib)^{-1}\\
&=&\frac{1}{2\pi\al}[(a+ib)^{-1}+(a-ib)^{-1}],
\eqast
and integrate
$
J_{-1}(\al,\Cp)=\frac{1}{2\pi i\al}\ln\frac{a+ib}{a-ib}.
$
Thus,
\bbe\label{Jm1}
J_{-1}(\al,\Cp)=\frac{1}{2\pi i\al}\ln\frac{\Cp}{\Cm}=\frac{\varphi_0}{\al\pi}.
 \ee
%where $\varphi_0=-\arctan(\be\tan(\nu\pi/2))$. 
The function  $F_{1}(x')$ can be calculated as the integral on the RHS of \eq{difp2}. 
The simplified trapezoid rule is
 \beqa\label{cpdfSL2y2}
 F_{1}(x')
 &=& -\frac{\ze }{\pi}\Im \left(\sum_{j=-N_-}^{N_+}f(j\ze)\right),
 \eqa
 where 
 $
 f(y)=e^{-ix'e^{i\om+y}-C_+e^{\al(i\om+y)}}-e^{-\Cp e^{\al(i\om+y)}}
 $. 
Parameters $\ga^\pm_0, \om, d$ are chosen as for
the pdf. We use an approximate bound for the Hardy norm  $H\le H_++H^0_++H_-+H^0_-$, where $H_\pm$ are
the same as in the case of the pdf, and $
H^0_\pm=(1+e^{x'\sin(\om\pm d)})e^{-c_\infty(\om\pm d)}.
$
As in  Sect.~ \ref{choice_ze}, we set $\ze=2\pi d/\ln(10 H/\eps)$. 
The  limits in the simplified conic
trapezoid rule are defined as follows. First, $N_-=\mathrm{ceil}\, (\ln(|x'|/(\eps\pi))/\ze)$.
 Next, the truncation parameter $\La_+=N_+\ze$, hence, $N_+$, is
determined from $\La_+=\ln\La_1$, where $\La_1$ satisfies
\bbe\label{La1tJ000con2}
%Err_{tr}(\La_1)=
 \frac{1}{\pi }\int_{\La_1}^{+\infty}\rho^{-1}\left(1+e^{x'\sin(\om)\rho}\right)e^{-c_\infty(\om)\rho^\al}d\rho<\eps.
\ee
Asymptotic expansions in the left truncated tail of the sum can be done similarly to the case of the pdf.
 As $y\to-\infty$,
 \bbe\label{fcpdfnuneq1}
 f(y)=-ix'e^{i\om}e^{y}+2R_2(y),
 \ee
 where
 \bbe\label{boundR2cpdfnuneq1}
 |R_2(y)|\le C(y)((x')^2 e^{2y}/2+|x'||\Cp|e^{(1+\al)y}),
 \ee
 and $C(y)\to 1$.
  Hence, we set
   $\La_{-,12}=\frac{\ln(|x'\Cp|/(\eps\pi))}{1+\al}$,  $\La_{-,11}=\ln(2 x'^2/(\eps\pi))/2$,
  %$\La_{-,22}=\ln(2/(|\Cp|^2\eps\pi))/(2\al)$,
 $\La_-=\min\{ \La_{-,11}, \La_{-,12}\},$ \\ % \La_{-,22}\}$,
 $N_-=\mathrm{ceil}\, (\La_-/\ze)$,
 calculate $S_{10}(\ze, N_-)$, 
 and add the term $-ix'e^{i\om}S_{10}(\ze, N_-)$ inside the brackets on the RHS of \eq{cpdfSL2y2}:
 \beqa \nonumber
%&& F_{1}(x')\\\nonumber
  F_{1}(x')&=&-\frac{\ze }{\pi}\Im \left(\sum_{j=-N_-}^{N_+}f(j\ze)-ix'e^{i\om}S_{10}(\ze, N_-)\right)\\\label{cpdfSL2y3}
&=& \frac{x'\cos(\om)\ze}{\pi}S_{10}(\ze, N_-)-\frac{\ze }{\pi}\Im\sum_{j=-N_-}^{N_+}f(j\ze).
 \eqa
 
 \subsection{The case of $\al< 1$ and large $|x'|$}\label{conic_cpdfnuLeq1}
 Let $c_0>0$ and $\om_0\in(-\pi/2,\pi/2)$. Similarly to \eq{J-1} and \eq{Jm1}, 
we calculate
\beqast\nonumber
J(c_0e^{i\om_0})&=&\frac{1}{2\pi}\int_{-\infty}^{+\infty}\frac{e^{-c_0|\xi|e^{i\om_0\mathrm{sign}\,\xi}}-e^{-|\xi|}}{-i\xi}d\xi
%\\\label{first_cpdf}&=&
=\frac{\om_0}{\pi}
\eqast
and obtain
 \beqa\label{cpdf_conic_nuLe1}
 F(x')&=&\frac{1}{2}+\frac{\om_0}{\pi}
+\frac{1}{\pi}\Re\int_0^{+\infty}\frac{e^{-ix'\xi-\Cp\xi^\al}-e^{-c_0e^{i\om_0}\xi}}{-i\xi}d\xi.
 \eqa
 We choose $\ga^\pm_0$, $\om$ and $d$ as in the case of the pdf,  
 and change the variable $\xi=e^{i\om+y}$: 
% \beqa\label{cpdfSL2yle11}
\beqast
F(x')&=&\frac{1}{2}+\frac{\om_0}{\pi}
 -\frac{1}{\pi}\Im \int_\bR
\left
(e^{-ie^{i\om}x'e^y-\Cp e^{i\al\om}e^{ay}}-e^{-c_0 e^{i(\om+\om_0)}e^y}\right)dy.
 \eqast
We set  $c_0=|x'|$, and pass to the limit $\om_0\to \mathrm{sign}\,x' \pi/2$:
\beqast\label{cpdf_conic_nuLe12}
 F(x')&=&\frac{1+\mathrm{sign}\,x'}{2}
 %\\\nonumber &-&
- \frac{1}{\pi}\Im \int_{-\infty}^{+\infty}f(y)dy,
 \eqast
 where  $f(y)=e^{-ie^{i\om}x'e^y}(e^{-\Cp e^{i\al\om}e^{ay}}-1)$.
Next, we apply the simplified trapezoid rule. 
 \bbe\label{cpdfSL2yle12}
F(x')=\frac{1+\mathrm{sign}\,x'}{2}
 -\frac{\ze}{\pi}\Im \sum_{j=-N_-}^{N_+}f(y_j).
 \ee
  The parameters $d, \ze, N_+$ are chosen as in the case of the pdf - see Sect.~ \ref{conicpdfnuLeq1},
 and the asymptotic expansions and corresponding choices of $N_-$ are modified in the straightforward fashion.
 In all cases, the operator $\Re$ is replaced by $-\Im$, the factor $e^{i\om+y}$ is absent, hence, the tail decreases slower, the coefficients 
 have to be multiplied by $e^{-i\om}$, and $S_{rs}(\ze,N_-)$, $S^k_n(\ze,N_-)$ are replaced with
 $S_{r-1,s}(\ze,N_-)$, $S^k_{n-1}(\ze,N_-)$.     We use
the Taylor expansion of the exponent of order 1 or 2. 

\sbr\noindent
{\em Ord 1.}  We define $\La_{11}=\ln(2|x'\Cp|/(\eps\pi))/(1+\al)$, 
$\La_{12}=\ln(|\Cp|^2/(\eps\pi))/(2\al)$, $\La_1=\max\{\La_{11},\La_{12}\}$,\\ $N_-=\mathrm{ceil}(\La_1/\ze)$,
calculate $S_{01}(\ze,N_-)$,
%\bbe\label{le1S1}
%S_1(N_-)=\frac{e^{-(1+\al)(N_-+1)\ze}}{1-e^{-(1+\al)\ze}},
%\ee
and then 
 \beqa\label{cpdf_conic_nuLe131}
 F(x')&=&\frac{1+\mathrm{sign}\,x'}{2}
 -\frac{\ze}{\pi}\Im\left(\sum_{j=-N_-}^{N_+}
 f(j\ze)-\Cp e^{i\om\al}S_{01}(\ze,N_-)\right).
 \eqa

\sbr\noindent
{\em Ord 2.}  We define $\La_{21}=\ln(3(x')^2|\Cp|/(2\eps\pi))/(2+\al)$, 
$\La_{22}=\frac{\ln(3|x'\Cp^2|/(2\eps\pi))}{1+2\al}$, 
$\La_{23}=\frac{\ln(|\Cp^3|/(2\eps\pi))}{3\al}$, \\
$\La_1=\max_j\La_{2j}$,  $N_-=\mathrm{ceil}(\La_1/\ze)$,
calculate $S_{01}(\ze,N_-)$, $S_{11}(\ze,N_-)$  and $S_{02}(\ze,N_-)$,
%\beqa\label{le1S21}
%S_{21}(N_-)&=&\frac{e^{-(2+\al)(N_-+1)\ze}}{1-e^{-(2+\al)\ze}}\\\label{le1S22}
%S_{22}(N_-)&=&\frac{e^{-(2\al+1)(N_-+1)\ze}}{1-e^{-(2\al+1)\ze}}
%\eqa
and then 
 \beqa\nonumber
 F_(x')&=&\frac{1+\mathrm{sign}\,x'}{2}
%\\\nonumber &-&
- \frac{\ze}{\pi}\Im\left(\sum_{j=-N_-}^{N_+}
 f(j\ze)-\Cp e^{i\om \al}S_{01}(\ze,N_-)\right.\\\label{cpdfSL2yLe12}
 &&+\left.ix'\Cp e^{i(1+\al)\om} S_{11}(\ze,N_-)
 +\frac{\Cp^2e^{2i\om\al}}{2}S_{12}(\ze,N_-)\right).
 \eqa

 \subsection{Further simplifications}\label{NuLe1further_cpdf}
If $|x'|+|\Cp|\cos(\varphi_0-\mathrm{sign}\,x' \al\pi/2)$ is positive and not very small, then we simplify \eq{cpdf_conic_nuLe12} letting
   $\om=-\mathrm{sign}\, x' \pi/2$;
  the number of elementary operations needed to calculate the individual terms in the simplified trapezoid rule
  decreases. The parameter $d_0$ needed to choose $\ze$ is determined as
  in Sect. \ref{NuLe1further}. Since there is no  sizable loss in the width of the strip of analyticity, the number of terms
  remains approximately the same, and the total CPU time decreases. Then we change the variable to obtain, for $\pm x'>0$,
 \beqa\label{cpdf_conic_nuLe13}
 F(x')&=&\frac{1\pm 1}{2}-
 \frac{1}{\pi}\Im\int_\bR\left(e^{-|x'|e^y-\Cp e^{\mp i\frac{\al\pi}{2}}e^{\al y}}-e^{-|x'|e^y}\right)dy.
 \eqa
 The second term under the integral sign is real and can be omitted: 
  \beqa\label{cpdf_conic_nuLe14}
F(x')&=&\frac{1\pm 1}{2}
 +\frac{1}{\pi}\int_\bR e^{-|x'|e^y-|\Cp| \cos(\varphi_0\mp \frac{\al\pi}{2})e^{\al y}}
 \sin(|C_+|\sin(\varphi_0\mp \frac{\al\pi}{2})e^{\al y})dy.
 \eqa
 The representation of
 the integrand in the form \eq{cpdf_conic_nuLe13} is convenient for the choice of the parameters
 of the simplified trapezoid rule (and justification of the application of this rule) because
 \[
 f(y)=e^{-|x'|e^y-\Cp e^{\mp i\frac{\al\pi}{2}}e^{\al y}}-e^{-|x'|e^y}
 \] decays fast as $y\to \infty$ remaining
 in a strip around the real axis, and the upper bounds are easy to obtain. Should we decide to write the
 imaginary part of the integrand explicitly, the derivation of the upper bounds would be not so straightforward.
 We leave to the reader the substitution  $\om=-\mathrm{sign}\, x' \pi/2$ in the coefficients 
 on the RHSs of \eq{cpdf_conic_nuLe131} and
 \eq{cpdfSL2yLe12}. %, \eq{cpdfSL2yLe13}.

\section{Exponential acceleration. CPDF, $\al=1$}\label{coniccpdfnu1}
We modify \eq{cpdf_conic_nuLe1}
 \beqa\label{cpdfconicnuneq1}
F(x')&=&\frac{1}{2}+J_{-1}(1,C_0)-\frac{1}{\pi}\Im \int_0^{+\infty}\frac{e^{-ix'\xi-\psi^0(1,\sg,\be,\xi)}-e^{-C_0\xi}}{\xi}d\xi,
\eqa
where $C_0=c_0 e^{-i\om}$, $c_0>0$, $\om$ is chosen as in the case of the pdf,
\[
J_{-1}(1,C_0)=\frac{1}{\pi}\Re\int_{0}^{+\infty}\frac{e^{-C_0 |\xi|}-1}{-i\xi}d\xi
\]
is calculated similarly to $J_{-1}(\al,\Cp)$ (see \eq{Jm1}): we differentiate the integral w.r.t. $C_0$,
calculate the derivative and integrate. The result is
%\bbe\label{Jm11}
\[
 J_{-1}(1,C_0)=\Re\frac{1}{\pi i}\ln C_0=-\frac{\om}{\pi}.
 \]
% \ee
The integral on the RHS of \eq{cpdfconicnuneq1} is calculated using
the simplified conic trapezoid rule:
\bbe\label{cpdfSL2ynu10}
 F(x')=\frac{1}{2}-\frac{\om}{\pi}-\frac{\ze }{\pi}\Im\sum_{j=-N_-}^{N_+}f(j\ze),
 \ee
 where 
 \[
f(y)= \exp[e^{i\om+y}(-(ix'+\sg(1-2\be\om/\pi))e^y -i(2\sg\be/\pi)ye^y))]
 \]
 (we take into account that 
$-\exp(-c_0 e^y)$ is real, hence, can be omitted).  As an approximate upper bound for the Hardy norm,
we take  $H\le 4+|f(-id)|+|f(id)|+H_++H_-$, where $H_\pm$ are the same as in the case of the pdf.
Then we set $\ze=2\pi d/\ln(10H/\eps)$. For simplicity, we use the same truncation parameters
$\La_{1,+}$ and $\La_+=\ln \La_{1,+}$, hence, $N_+$, as in the case
of the pdf. However, if $\La_{1,+}$ is large, this prescription yields an unnecessary large $N_+$.
For large $\La_{1,+}$, a simple fairly accurate improvement can be obtained as follows: find $\La_{1,+}=\La_{1,+}(\eps)$
using the procedure for the pdf, then reassign $\eps:=\eps\La_{1,+}$, and apply the prcedure
for the pdf with the new $\eps$.  Finally,
 set $\La_+=\ln \La_{1,+}$,  $N_+=\mathrm{ceil}\,(\La_+/\ze)$.

 We can use \eq{cpdfSL2ynu10} with $N_-=\mathrm{ceil}\,(\La_-/\ze)$, where
$\La_-=\ln(1/\eps)$, which requires an unnecessary large $N_-$. To decrease $N_-$, we use the asymptotic expansion
%\beqa\label{deffcpdfnu1}
%f(y)&=&e^{-ix'e^{i\om +y}-\sg e^{i\om +y}(1+i(2\sg\be/\pi)(y+i\om))}-e^ {-c_0e^{y}}\\
%\beqa\label{fcpdfnu1}
\beqast
f(y)&=&-((ix'+\sg(1-\frac{2\be\om}{\pi})e^{i\om}-c_0)e^{y}+i\frac{2\sg\be}{\pi}e^{i\om}e^{y}(-y)
+C(y)e^{2y}R_2(y),
 \eqast
 where $C(y)\to 1$ as $y\to-\infty$, and $R_2(y)$ admits the upper bound \eq{boundR2pdfnu1}.
 We write the RHS of \eq{boundR2pdfnu1} as $a_0+a_1y+a_2y^2$, define $\La_{-,1}=\ln(3a_0/(\eps\pi))/2$,
 find $\La_{-,j}$ as
 an approximate solution to $e^{-2\La}\La^j=\eps\pi/(3a_j)$, $j=1,2$, set $\La_-=\max\La_{-,j}$, $N_-=\mathrm{ceil}\,(\La_-/\ze)$.
We obtain the formula similar to \eq{cpdfSL2ynu10},
with two additional terms:
\beqa\label{cpdfSL2ynu1} 
 F(x')&=&\frac{1}{2}-\frac{\om}{\pi}
 -\frac{\ze }{\pi}\Im
 \left(\sum_{j=-N_-}^{N_+}f(j\ze)-A_0S_{10}(N_-)+A_1S^1_1(N_-)\right),
\eqa
% \\\label{pdfSL2ynu12}
% && \hskip0.5cm \left.\left.-((ix'+\sg(1-2\be\om/\pi)e^{i\om}-c_0)S_0(N_-)+\frac{2i \sg \be\ze}{\pi} \right)\right).
 %\eqa
 where $A_0=((ix'+\sg(1-\frac{2\be\om}{\pi})e^{i\om}-c_0), A_1=i e^{i\om}\ze \frac{2\sg\be}{\pi}$.

If $x'\be>0$, \eq{cpdfSL2ynu1} can be simplified 
 letting $c_0=|x'|$ and passing to the limit
$\om\to -\mathrm{sign}\, x'(\pi/2)$.   We obtain
 \beqast\label{cpdfSL2ynu1xpgen}
 F(x')&=&\frac{1+\mathrm{sign}\, x'}{2}
 - \frac{\ze}{\pi}\left(\Im\sum_{j=-N_-}^{N_+}f(j\ze)+\mathrm{sign}\,x' \sg(1+\be)S_{10}(\ze,N_-)\right),
\eqast
where $f(y)=\exp[(i\sg (1+\be)-x')e^y-(2\sg\be/\pi)e^y y].$ Below,  we consider the case $x'>0, \be>0$;
the case $x'<0, \be<0$ is by symmetry.

 Using the second order Taylor expansion of the
exponential function, we derive from \eq{cpdfSL2ynu1}
%\beqa\label{cpdfSL2ynu1p2}
\beqast
 F(x')&=&\frac{1+\mathrm{sign}\, x'}{2}
% \\\nonumber&&
- \frac{\ze }{\pi}\left(\Im\sum_{j=-N_-}^{N_+}f(j\ze)
+\sg(1+\be)S_{10}(\ze,N_-)-x'\sg(1+\be)S_{20}(\ze, N_-)
\right.\\\nonumber
 &&\left.
 %\\\nonumber &&
 +\frac{2\sg^2\be(1+\be)\ze}{\pi}S^1_2(\ze,N_-)\right).
\eqast
Here $N_+$ is as above, and $N_-=\mathrm{ceil}\,(\La_-/\ze)$,
where $\La_-=\max\{\La_{1,-}, \La_{2,-}\}$,  $\La_{1,-}=\log(B_1/(4\eps\pi))/3$, \\ $B_1=(4/3)|ix'+\sg(1+\be)|^3$,
and $\La_{2,-}$ is the solution of the equation $e^{-3\La}\La^3=\eps\pi/B_2$, where $B_2=(4/3)(2\be\sg/\pi)^3$.

\section{Derivatives of PDF and unification}\label{der_and_quantiles}

\subsection{Evaluation of $p'$} The methodology is essentially the same as in the case of the pdf.
Calculation of the terms of the asymptotic expansion for the integral on the RHS of \eq{derpdfSL} is essentially the same
(an additional factor $-i\xi$ appears), and the residual term is calculated using the same
$\ga^\pm_0$, $\om$, $d$. The bound for the truncation error requires a marginally larger $\La$ due to an additional factor $-i\xi$ which increases at infinity. A small increase in $N_+$ is compensated by a sizable
decrease in $N_-$. 

\subsection{Evaluation of $\partial_\be p(\al,\sg,\be, x')$, $\partial_\sg p_{st}(\al,\sg,\be, x')$ 
and $\partial_\al p(\al,\sg,\be, x')$} If $\al\neq 1$, then
\beqa\nonumber
\partial_\be p(\al,\sg,\be, x')
&=&\frac{\sg^\al\tan(\al\pi/2)}{\pi}\Re \int_0^{+\infty}i\xi^\al e^{-ix'\xi-\Cp(\al,\sg,\be)\xi^\al}d\xi\\\label{derpbe}
&=&-\frac{\sg^\al\tan(\al\pi/2)}{\pi}\Im \int_0^{+\infty}\xi^\al e^{-ix'\xi-\Cp(\al,\sg,\be)\xi^\al}d\xi.
\eqa
If $\al=1$, then
\beqa\nonumber
\partial_\be p(1,\sg,\be, x')
&=&\frac{2\sg/\pi}{2\pi}\Re \int_0^{+\infty}(-i\xi \ln\xi)e^{-ix'\xi-\psi^0_{st}(1,\sg,\be,\xi)}d\xi\\\label{derpbenu1}
&=&\frac{\sg}{\pi^2}\Im \int_0^{+\infty}\xi \ln\xi e^{-ix'\xi-\psi^0_{st}(1,\sg,\be,\xi)}d\xi.
\eqa
The integrals on the RHS of \eq{derpbe} and \eq{derpbenu1} are calculated using the    simplified conic trapezoid rule.
The parameters $\ga^\pm_0, \om, d$ are as in the case of the pdf, and the bounds for the Hardy norm and truncated tails are modified in
the evident manner. The resulting $\ze$ and $N_-$ are smaller,  and $N_+$ larger than the ones for the pdf.
The derivatives $\partial_\sg p(\al,\sg,\be, x')$ 
and $\partial_\al p(\al,\sg,\be, x')$ are calculated in a similar fashion.

\subsection{Unified schemes} If $\al\in (0,1)$, then, in all the integrals above, one can justify the change of variables $ix'\xi=y$ 
and reduce the calculations to an integral of the form
\bbe\label{universal}
I(v, v_1; a, A)=A\int_0^{+\infty} y^{-v_1}e^{-y-ay^v}dy,
\ee
where $v=\al\in (0,1)$, $v_1, a, A\in \bC$. If $\al\in (1,2)$, then, in all the integrals above, one can justify the change of variables $\Cp\xi^\al=y$, and reduce to 
 the calculations to an integral of the form
\eq{universal}, where $v=1/\al\in (0,1)$, $v_1, a, A\in \bC$.
 Hence, one can design a general procedure for all integrals in the case $\al\neq 1$.
We will study this possibility in the future. 
%The schemes designed above admit fairly straightforward modifications.
%however, the design of a universal and stable set of recommendations
 %for the parameter choice is far from straightforward. Possibly, the case-by-case
%approach used above is more convenient for applications. 

If the exponential change of variables is inefficient,  it may be advantageous to use the polynomial or sub-polynomial acceleration.
If $\al<1$ is small, then both require an unnecessary large number of terms. The number of terms can be made much smaller using the preliminary
changes
of variables: either $\Cp\xi^\al=i\xi_1$, or $\Cp\xi^\al=-i\xi_1$, where $\xi_1$ runs over $\bR_+$. The first one is applied if $x'<0$ and $\be>0$,
and the second one is applied if $x'>0, \be<0$. In both cases, the change of variables is possible only if $\al\in (1/3,1)$
The resulting integrals can be formally interpreted as the ones for stable L\'evy processes of index $\al\in (1,3)$. Naturally,
stable L\'evy processes of order $\al\ge 2$ do not exist, but the integrals that define the pdf and cpdf are well-defined for any $\al>0$,
hence, there is no contradiction in this approach.

\section{Quantiles and Monte Carlo simulations}\label{algquantile}  
We consider evaluation of quantiles $x_a$, that is, solution
of the equation $F(x)=a$, where $a\in (0,1)$ and $F$ is the cumulative
distribution function; once an efficient procedure for quantile evaluation is available,
the procedure can be used for the Monte Carlo simulations.

If $Z$ is any random variable with continuous distribution, one can simulate $Z$ sampling a uniformly distributed random variable $U$ on $(0,1)$ and calculate $F^{-1}(U)$, where $F$ denotes the cumulative distribution function of $Z$. When an explicit formula for $F^{-1}$ is not available, it becomes important to be able to calculate its values very quickly and sufficiently accurately. A straightforward approach that was used with a limited success\footnote{The tails of the distributions decay too slowly, hence, the Monte Carlosimulations are moderately efficient only
if the index of the process is close to 2, and the distribution does not differ much from the normal distribution, with the exception
of far parts of the tails, which can be safely disregarded in this case.}  for simulation
of stable L\'evy processes is  as follows. One tabulates
the values of $F$ on a sufficiently long and fine grid of points $x_0,x_1,\dots,x_M$ on the real line and approximates $F^{-1}$ using linear interpolation.
This method is very attractive from the practical viewpoint, because the values $F(x_i)$ only have to be calculated once, and afterward the computational cost of each simulation of $Z$ is extremely low: one has to draw a sample $a$ of $U$, find $j$ satisfying $F(x_j)\leq a<F(x_{j+1})$ (which requires about $\log_2(M)$ comparisons)
and perform 4--5 arithmetic operations required for linear interpolation to find $x_a$.
%\begin{equation}\label{linintnon-log}
%x=x_j+(x_{j+1}-x_j)(a-F_j)/(F_{j+1}-F_j).
%\end{equation}
If $a<F_1$, one assigns $x_a=x_1$, and if $a>F_M$, then
one assigns $x_a=x_M$.

In the application to the Monte Carlo simulations, this method  has 3 sources of errors:
\begin{enumerate}[(1)]
\item
truncation error;
\item
errors of linear interpolation;
\item
errors of evaluation of $F_k$.
\end{enumerate}
The simplified conic trapezoid rule allows us to calculate $F_k=F(x_k)$ very accurately and fast;
if $\al\in (0,1)$ is not too close to 1, then the calculations are especially fast for $x_k$ large in absolute value.

Numerical experiments demonstrate that, given the parameters of the distribution,
one can use the same grid of a moderate size (150-300)
for calculations in a very large region in the tail of interest; furthermore, the parameters that define the grid
vary slowly as the parameters of the distribution vary. Hence, one can precalculate the expression in the exponent (bar the factor $e^{-ix' \xi(y)}$),
and the factor(s) outside the exponential sign
needed for the calculation of the pdf and cpdf
at points of an appropriate grid (we suggest to call these precalculated arrays {\em conformal principal components})
and use these arrays to calculate the pdf and cpdf very fast for $x$ that will appear in the iteration procedure
for the calculation of the quantile. In addition, one needs to precalculate several scalars used in the correction terms.

Thus, it is unnecessary to truncate the state space. Instead, it suffices to store the array of values
in a region $F(x)\in [0.001, 0.999]$ (or $F(x)\in [0.01,0.99]$) and use the array and an interpolation procedure 
if a simulated $a\in [0.001,0.999]$
(resp., in $[0.01,0.99]$). 
The conformal principal components
are used when a simulated $a$ is outside this region.    Note that a different set
of conformal principal components can be used for  fast tabulation of $p$ and $F$ in the region $[0.001,0.999]$
(resp., $[0.01,0.99]$). 

Below, we give an explicit scheme for calculation quantiles in the left tail, for processes of index $\al\in (0,1)$. We assume that the quantiles of interest
are in a region of $x_a$ such that after a rescaling $\xi\mapsto 10^{N_{sc}}\xi$, $x'/(\Cp \sigma^\al)$ is not small so that the simplest choice $\om=\pi/2$ is possible.
The scaling parameter $N_{sc}$ should be moderate, e.g., in the region $[-2,10]$. 

Assume that, for a given $a$,  one knows an interval $[x_-, x_+]\subset (-\infty, 0)$ where
$x_a$ is. As the numerical examples shown in Tables 12-14 demonstrate, efficient calculations are possible even if the interval is very wide. 
Assuming that the simplified conic trapezoid rule  is applicable for $x_+$, with $\om=\pi/2$, the same rule is applicable for all $x<x_+$. Typically, for $x<x_+$, $\ze$ and $\La_+$ are smaller than for $x_+$, and $\La_-$ larger.
To calculate the conformal principal components which can be used to evaluate $p(x)$ and $F(x)$ for all $x\in [x_-, x_+]$, we must use one set of parameters of the simplified trapezoid rule.

%\sbr\noindent
\paragraph{Algorithm for calculation of quantiles in the left tail, for processes of order $\al\in (0,1)$. }
\begin{enumerate}[1.]
\item Set $\om=\pi/2, d=k_d\om$.
\item Choose $N_{sc}$, $ord$ and a small error tolerance $\eps$.
\item
Using the recommendations for cpdf, choose $\ze$ and $\La_-$ for $x_-$, and  $\La_+$ for $x_+$. 
\item
To increase the accuracy of calculations at a small cost in the CPU time, choose $k_\ze\in [1.2, 1.5], k_{\La_-}\in [1.3, 1.5], k_{\La_+}\in [1, 1.15]$,
and reassign $\ze:=\ze/k_\ze$, $\La_-:=\La_-k_{\La_-}$ and $\La_+=\La_+k_{\La_+}$ if $\La_+\ge 0$, and $\La_+=\La_+/k_{\La_+}$
\item
Set $N_\pm=\mathrm{ceil}\, (\La_\pm/\ze)$.

\item
Calculate and store arrays $\vec y=\ze*(N_-:1:N_+)$, $E_y=\exp(\vec y)$, $EC_\al=\Re (\Cp e^{i\om \al})*\exp(\nu*\vec y)$,  
$SIN_\al=\sin(\Im (\Cp* e^{i\om \al})*\exp(\al*\vec y))$.
\item
Calculate and store scalars $S_{jk}(\ze, N_-)$ needed to correct the truncation errors
of the simplified trapezoid rule for pdf and cpdf, and the correction terms independent of $x'$.
\item
If the Newton method is used, write the function $x'\to (p_0(x'), F_0(x'))$
\begin{eqnarray}\label{itmneg}
Int&=&\exp(x'*E_y-EC_\al).*SIN_\al\\\label{Intnegtailcpdf}
F_0(x')&=&(\ze/\pi)*sum (Int)\\\label{negtailcpdf}
p_0(x')&=&(\ze/\pi)*sum (E_y.*Int).
\end{eqnarray}
At each step of the iteration procedure, use this function and add correction terms.
\item
If the bisection method is used, write the function $x'\to F_0(x')$.
At each step of the iteration procedure, use this function and add correction terms.
\end{enumerate}

\section{Completely asymmetric stable L\'evy processes}\label{one-sided}
If $X$ is completely asymmetric, that is, $\be=-1$ or $\be=1$, equivalently, either $\cp=0$ or $\cm=0$, then the characteristic exponent
admits analytic continuation to the complex plane with the cut along of one of the imaginary half-axis, hence, the sinh-acceleration can be
applied. By the  symmetry argument, it suffices to consider the case $\be=-1$. In this case, there exists
$a\ge 0$ s.t. $p(x')=0$ for $x'>a$.

\paragraph{Case $\al\neq 1$, $\be=-1$}\label{sinhnune1} 
We choose the parameters of the sinh-acceleration so that
$\om_1+b\sin\om'<0$ for any $\om'\in[-\om,\om]$.  Deforming the contour of integration,
making the change of variable $\xi=i\om_1+b\sinh(i\om+y)$,  we obtain
\beqast
p(x')&=&be^{\om_1x'}p^{norm}(x'),
\eqast
where 
\beqast
p^{norm}(x')&=&\frac{1}{2\pi}\int_\bR f(y)dy,\\
f(y)&=&\cosh(i\om+y)f_1(y)\\
f_1(y)&=&e^{-ix'b\sinh(i\om+y)+c_-\Ga(-\al)(-\om_1+ib\sinh(i\om+y))^\al}.
\eqast
The parameters $\om, b, \om_1$ are chosen as follows.
We choose $\gam_0$ and $\gap_0$ as in Sect.~\ref{choice_cone} for the case
 $x'<0$. We set $\om=(\gam_0+\gap_0)/2$, $d_0=(\gap_0-\gam_0)/2$. 
 In order that the Hardy norm be finite and not exceedingly large, it is necessary
 that the integrand in the $y$-coordinate decays fast as $y\to\infty$ along the boundaries $\Im y=\pm d$
  of the trip of analyticity. Hence, we take $k_d<1$,
 e.g., $k_d\in [0.8,0.95]$, and set $d=k_dd_0$.  Next, take $\om_1<0$, set $b_0=-\om_1/\sin(\om)$,
 choose $k_b\in [0.8,0.95]$ and set $b=k_bb_0$. 
  
 If $\eps>0$ is the error tolerance for $p(x')$, then $\eps_1=\eps e^{-\om_1x'}$ is the error
tolerance for $p^{norm}(x')$. We calculate $H_\pm$ as in Sect.~\ref{choice_ze}, and set $H=|f(-d)|+|f(d)|+H_++H_-$,
$\ze=2\pi d\ln (10H/\eps_1)$. We calculate $\La_1$ as
 in Sect.~\ref{choice_N_+}, for the error tolerance $\eps_1$, and set $\La=\ln(2\La_1/b)$, $N=\mathrm{ceil}\,(\La/\ze)$.
 Then we apply the simplified trapezoid rule
 \bbe\label{sinhpdf}
 p^{norm}(x')=\frac{\ze}{\pi}\Re\left(0.5 f(0)+\sum_{1\le j\le N}f(j\ze)\right).
 \ee
Since we have to use $\om_1<0$, and $b<-\om_1/\sin(\om)$,
it follows that for the calculations 
in the left tail, a very large number of terms in the simplified trapezoid rule is needed. 
Hence, if $x'$ is very large in the absolute value, we have to use a very small $\om_1$, and then $N$
is large. 
 
 For the cpdf, the calculation is similar
\beqast
F(x')
&=&\frac{be^{\om_1x'}}{2\pi}\int_\bR f(y)dy,
\eqast
where
\beqast
f(y)&=&\frac{\cosh(i\om+y)}{-\om_1+ib\sinh(i\om+y)}f_1(y),\\
%&&\cdot e^{-ix'b\sinh(i\om+y)+c_-\Ga(-\al)(-\om_1+ib\sinh(i\om+y))^\al}.
\eqast
and $f_1$ is the same as in the case of pdf.
The choice of the parameters of the sinh-acceleration is a modification
of the choice in Sect.\ref{coniccpdfnu1} similar to the modification in the case of pdf. The simplified trapezoid rule
is
 %\bbe\label{sinhcpdf}
 \[
 F(x')=be^{\om_1x'}\frac{\ze}{\pi}\Re\left(0.5 f(0)+\sum_{1\le j\le N}f(j\ze)\right).
 \]
  %\ee

\paragraph{Case $\al=1$, $\be=-1$}\label{sinhnu1}
Deforming the contour of integration, and
changing variable $\xi=i\om_1+b\sinh(i\om+y)$,  we obtain
\beqast
p(x')&=&\frac{be^{\om_1x'}}{2\pi}\int_\bR \cosh(i\om+y) e^{\psi(y)}dy,
\eqast
where
\beqast
\psi(y)&=& -ib\sinh(i\om+y)x'
+i(i\om_1+b\sinh(i\om+y))c_-\ln(-\om_1+ib\sinh(i\om+y))).
\eqast
For the cpdf, the calculation is similar, and the choice of the parameters is modified in the same manner
as in the case $\al\neq 1$.

\section{Polynomial acceleration}\label{s:Polynom}
For the sake of brevity, we apply polynomial acceleration in the case $\al\neq 1$ only.
\subsection{Preliminaries}\label{ss:prelim}
Using \eq{pdfstable}, \eq{cpdf2}, \eq{JRN2}, \eq{Jm1}, we have
\beqa\label{Pp}
p(x')&=&p(0)+\frac{1}{\pi}\Re\int_0^{+\infty} (e^{-ix'\xi}-1)e^{-\Cp\xi^\al}d\xi,\\\label{PF}
F(x')&=&\frac{1}{2}+\frac{\varphi_0}{\pi\al}+\frac{1}{\pi}\Re\int_0^{+\infty}\frac{e^{-ix'\xi}-1}{-i\xi}e^{-\Cp\xi^\al}d\xi.
\eqa
Polynomial and sub-polynomial accelerations can be advantageous only in  cases when
the exponential acceleration may require very large number of terms of the simplified trapezoid rule.
These cases are
\vskip0.1cm
\noindent
{\em Case $(++)$} $\al\in (0,1), x'>0, \be\in (0,1]$ (\& $\varphi_0<0$).
\vskip0.1cm
\noindent
{\em Case $(--)$} $\al\in (0,1), x'<0, \be\in [-1,0)$ (\& $\varphi_0>0$).
\vskip0.1cm
\noindent
{\em Case $(-+)$} $\al\in (1,2), x'<0, \be\in (0,1]$ (\& $\varphi_0>0$).
\vskip0.1cm
\noindent
{\em Case $(+-)$} $\al\in (1,2), x'>0, \be\in [-1,0)$ (\& $\varphi_0<0$).
\vskip0.1cm
In Cases $(-+)$ and $(--)$, we make the change of variables \eq{pol1/2m}, and in Cases $(+-)$ and $(++)$,
 the change of variables \eq{pol1/2p}.  After that, in all cases, we make the change of variables \eq{yPchange}.
 The choice of the parameters $a, m, b$ is determined by the requirement that, in the $y$-coordinate,
 both of the factors $e^{-ix'\xi(\eta(y))}$ and $e^{-C_+\xi(\eta(y))^\al}$ decay in any strip $S_{(-d,d)}$, $d\in (0,1)$ as $y\to\infty$ remaining 
 in the strip;
 it is important that the rate of decay is as large as possible. Note that if $\al<1$ is sizably smaller than 1,
 then it is advantageous to make a preliminary change
 of variables and reduce to the case $\al>1$. The reduction can be justified if $|\varphi_0|<\pi(\al-0.5)$; this condition implies that $\al>1/3$.

%\subsection{Case $(-+)$} We make the changes of variables \eq{pol1/2m}, \eq{yPchange}.
\begin{lem}\label{lem:stripypol} 
For any $m\ge 1$, there exists $b>0$ such that  functions $y\mapsto \eta(y)$ and  $y\mapsto \xi(\eta(y))$ admit analytic continuation
to $S_{(-1,1)}$.
\end{lem}
\begin{proof} The analyticity of $y\mapsto\eta(y)$ is evident. To prove the analyticity of $y\mapsto \xi(\eta(y))$,
it suffices to show that 
%\bbe\label{Stripy1}
\[
1\pm ib(0.5(y+\sqrt{1+y^2}))^m\not\in (-\infty,0],\ y\in S_{(-1,1)},
\]
equivalently, $\pm ib(0.5(y+\sqrt{1+y^2}))^m\not\in (-\infty,-1]$, equivalently,
\bbe\label{Stripy2}
y+\sqrt{1+y^2}\not\in 2b^{-1/m}e^{\mp i\pi/(2m)}[1,+\infty], \ y\in S_{(-1,1)}.
\ee
As $y\to\infty$ remaining in the strip so that $\Re y\to-\infty$, the LHS of \eq{Stripy2} tends to 0,
and as  $y\to\infty$ remaining in the strip $S_{(-1,1)}$ so that $\Re y\to+\infty$, the imaginary part of the LHS of \eq{Stripy2} tends to 0.
Hence, there exists $L$ such that \eq{Stripy2} holds if $|\Re y|\ge L, \Im y\in [-1,1]$. But $\{y+\sqrt{1+y^2}\ | \ |\Re y|\le L, |\Im y|\le 1\}$ is a compact,
and, given a compact $K\subset \bC$, $2b^{-1/m}e^{\mp i\pi/(2m)}[1,+\infty]\not\in K$, if $b>0$ is sufficiently small.
\end{proof}
\begin{rem}{\rm It is easy to show that if $m=1$, then $b=2$ is admissible.}
\end{rem}
If $x'<0$ (resp., $x'>0$), we make the change of variables \eq{pol1/2m} (resp.,  \eq{pol1/2p}). In both cases,  as $\Re y\to+\infty$ 
and $y$ remains in the strip $S_{(-1,1)}$, we have
\[
\xi(\eta(y))\sim e^{-\mathrm{sign} x' i(a-1)\pi/2}b^a\eta(y)^{ma},
\]
hence, if $a\in (1,3)$, then
\[
\Re(ix'\xi(\eta(y)))\sim -|x'|\Re e^{\mp ia\pi/2}|\eta(y)|^\al\to+\infty
\] 
Set $\bar a=1+1/\al-2|\varphi_0|/(\pi\al)$. If $\al\in (1,2)$, then $\varphi_0\in (-\pi/2,0]$, 
and if $\al\in (0,1)$, then $\varphi_0\in [-\al\pi/2,0]$ with the equality $\varphi_0=-\al\pi/2$ possible only if $\be=1$.
Thus, $\bar a>1$ if either $\al\in (1,2)$ or $\be<1$. 
\begin{lem}\label{lem:choice_a} Let $\bar a>1$ and $a\in [1,\min\{3,\bar a\})$. Then 
  $\Re (C_+e^{\xi(\eta(y))^\al})\to +\infty$ as $\Re y\to+\infty$ and $y$ remains in the strip
$S_{(-1,1)}$.
\end{lem}
\begin{proof}
If $x'>0$, hence, $\varphi_0<0$ and \eq{pol1/2p} is used, an equivalent condition is
$\Re e^{i\varphi_0}\left(ie^{-ia\pi/2}\right)^\al>0$ $\Leftrightarrow \varphi_0+(\pi/2-a\pi/2)\al\in (-\pi/2,\pi/2) \Leftrightarrow
2\varphi_0/(\pi\al)+1-a\in(-1/\al,1/\al)$ $\Leftrightarrow a\in (1-1/\al+2\varphi_0/(\pi\al), 1+1/\al+2\varphi_0/(\pi\al))$.
Hence, the necessary and sufficient condition on $a$ is
$a\in [1,\min\{3, 1+1/\al+2\varphi_0/(\pi\al)\})$. Similarly, if $x'<0$, hence, $\varphi_0>0$ and \eq{pol1/2m} is used, an equivalent condition
is
$a\in [1,\min\{3, 1+1/\al-2\varphi_0/(\pi\al)\})$.
\end{proof}
Set $\bar\varphi_0=\varphi_0-\mathrm{sign} x' (a-1)\al\pi/2$, and
$
c_\infty:=c_\infty(a):=|\Cp|\cos\bar\varphi_0.
$
We choose $a\in [1,\min\{3,\bar a\})$ equal to 1 or close to 1 so that $c_\infty:=c_\infty(a)>0$
is not too small.

\subsection{Choice of $\ze$ and $N_+$}\label{ss:pol_choices of ze and N_+} 
Choose $d\in (0,1)$, e.g., $d=0.8$, and,
if $x'<0$, set $\bar\varphi_0=\varphi_0+(a-1)\al\pi/2$;  if 
$x'>0$, set $\bar\varphi_0=\varphi_0+(1-a)\al\pi/2$. Set $c_\infty=|\Cp|\cos\bar\varphi_0$. 
Denote by $f(y)$ and $f_1(y)$ the integrands in \eq{Pp} and \eq{PF}, in the $y$-coordinate.
 As $\Re y\to+\infty$ and $y$
remains in the strip $S_{(-1,1)}$, $f(y)$ and $f_1(y)$  are
$O(e^{-c_\infty |y|^{ma\al}})$, where the constant in the $O$-term is 1 and $<1$, respectively. Hence,
the Hardy norm admits an approximate bound via 
\beqast
H&=&\frac{1}{\pi}\left(|f(-id)|+|f(id)|+2\int_0^{+\infty} e^{-c_\infty \rho^\al}d\rho\right)\\
&=& \frac{1}{\pi}\left(|f(-id)|+|f(id)|+2c_\infty^{-1/\al}\Gamma(1/\al+1)\right).
\eqast
We set $\ze=2\pi d/\ln(10H/\eps)$. 
The reader observes that if $\al>0$ is small, $H$ can be very large and $\ze$ very small; a similar integral appears
below, when we derive an approximation to $N_+$. This explains why it is advantageous to reduce
to the case $\al>1$ if the initial $\al$ is too small; but this reduction is possible only if $\al\in(1/3,1)$.

We find $\La_+$ as an approximate solution of 
\[
\int_{\La_1}^{+\infty}e^{-c_\infty \rho^\al}d\rho=\eps\pi,
\]
and set $\La_+=b^{-1/m}(\La_1)^{1/(am)}$, $\La_+=\mathrm{ceil}\, (\La_+/\ze)$.

\subsection{Asymptotics as $y\to-\infty$ and choice of $N_-$}\label{ss:asymp_choice of N_-}
As $ y\to -\infty$, 
\beqast
\eta(y)&=&0.5(y+\sqrt{1+y^2)}\\
&=& 0.5(y+(-y)(1+(-y)^{-2})^{1/2})\\
%&=& 0.5 \left(\frac{1}{2(-y)}-\frac{1}{8(-y)^3}+\cdots\right)\\
&=& \frac{1}{4(-y)}\left(1-\frac{1}{4(-y)^2}+\cdots\right),
\eqast
and as $\xi\to 0$,
\beqast
\left(e^{-ix'\xi}-1\right)e^{-\Cp\xi^\al}&=&-ix'\xi\left(1-\frac{ix'\xi}{2}-\Cp \xi^\al+\cdots\right)\\
\frac{e^{-ix'\xi}-1}{-i\xi}&=&x'\left(1-\frac{ix'\xi}{2}-\Cp \xi^\al+\cdots\right).
\eqast
%Furthermore, 
\paragraph{If the change of variables \eq{pol1/2m} is made,} then, as $\eta\to 0$,
\beqast
\xi(\eta)&=&i(1-(1+ib\eta^m)^a)\\
&=& -i\left(iab\eta^m-\frac{a(a-1)}{2}b^2\eta^{2m}+\cdots\right)\\
&=& ab\eta^m\left(1+i\frac{(a-1)b}{2}\eta^{m}+\cdots\right),
\eqast
and, therefore, as $y\to-\infty$,
\beqast
\xi(\eta(y))&=&\frac{ab}{4^m(-y)^m}\left(1-\frac{m}{4(-y)^2}+\cdots\right)\left(1+\frac{i(a-1)b}{2}\left(\frac{1}{4^m(-y)^m}\right)^2+\cdots\right)\\
&=&\frac{ab}{4^m(-y)^m}\left(1-\frac{m}{4(-y)^2}+i\frac{(a-1)b}{2\cdot 4^m (-y)^{m}}+\cdots\right)
\eqast
and
\beqast
\frac{d\xi(\eta(y))}{dy}&=&\frac{mab}{4^m(-y)^{m+1}}
\left(1-\frac{m+2}{4(-y)^2}+i\frac{(a-1)b}{4^m(-y)^m}+\cdots\right).\eqast
Thus, in the $y$-coordinate, the integrand in \eq{Pp}, denote it $f(y)$,  
has the following asymptotics as $y\to-\infty$:
\beqast
f(y)&=&-i\frac{x'(ab)^2}{4^{2m}(-y)^{2m+1}}\left(1-\frac{2(m+1)}{4(-y)^2}%\right.\\&+&\left.
+ i\frac{b(-x'a+3(a-1)/2)}{4^m(-y)^m}-\Cp\frac{(ab)^\al}{4^{m\al}(-y)^{m\al}}+\cdots\right),
%\left(1-\frac{m}{4(-y)^2}+\cdots\right)\right)\\
%&& \cdot\left(\frac{mab}{4^{m}(-y)^{m+1}}\left(1-\frac{m+2}{4(-y)^2}+\cdots\right)\right),
\eqast
and its real part has the asymptotics 
\beqast
\Re f(y)&=&C^0_f(-y)^{-3m-1}+C^\al_f(-y)^{-(2+\al)m-1}+O((-y)^{-\bar m-3}),
\eqast
where $\bar m=m(2+\min(1,\al))$, and
\beqast
C^0_f&=&\frac{x'm(ab)^2}{4^{3m}}b\left(-x'a+\frac{3(a-1)}{2}\right),\\
C^\al_f&=&-\frac{x'm(ab)^{2+\al}}{4^{m(2+\al)}}|\Cp|\sin\varphi_0.
\eqast
Therefore, in the case of pdf, if no leading term is separated, we define
\[
\La_-=\max\left\{\left(\frac{|C^0_f|}{\eps\pi 3m}\right)^{1/(3m)}, \left(\frac{|C^\al_f|}{\eps\pi m(2+\al)}\right)^{1/(m(2+\al))}\right\}
\]
set  $N_-=\mathrm{ceil}\,(\La_-/\ze)$, and calculate
\[
p(x')=p(0)+\frac{\ze}{\pi}\Re\sum_{j=-N_-}^{N_+}f(j\ze).
\]
Note that when calculating $f(y)$, we use 
\[
\Cp \xi^\al=\Cp e^{i\al\pi/2}(-1+(1-ib\eta^m)^a)^\al;
\]
with the choice $a=1, b=1$, $m=2$, the calculations are faster than with the other choices:
$\Cp \xi^\al=\Cp \eta^{2\al}$.

In the case of cpdf, the real part of the integrand in \eq{PF} in the $y$-coordinate, denote it $\Re f_1(y)$, has the asymptotics
\beqast
\Re f_1(y)&=&C^0_{f1}(-y)^{-m-1}+C^1_{f1}(-y)^{-m-3}+C^{1\al}_{f1}(-y)^{-2m-1}
+C^\al_{f1}(-y)^{-(\al+1)m-1}\\
&&+O((-y)^{-\tilde m-1}),
\eqast
where $\tilde m=\min\{2(m+1), m(1+\al)+2, m+4\}$, 
\beqast
C^0_{f1}&=&x'abm 4^{-m},\\
C^1_{f1}&=&-x'abm(m+2) 4^{-m-1},\\
C^\al_{f1}&=&-x'm(ab)^{1+\al}|\Cp|\cos\varphi_0 4^{-m(1+\al)},\\
C^{1\al}_{f1}&=&\frac{(x')^2a(a-1)b^2}{4^{2m}}.
\eqast 
Therefore, in the case of cpdf, if no leading term is separated, we define
\[
\La_-=\left(|x'|ab/(\eps\pi 4^m)\right)^{1/m},\ N_-=\mathrm{ceil}\,(\La_-/\ze).
\]

\paragraph{If the change of variables \eq{pol1/2p} is made,} then
\[
\xi(\eta)=ab\eta^m\left(1-i\frac{a(a-1)}{2}b^2\eta^{m}+\cdots\right),
\]
and the calculations above modify in a straightforward fashion (it suffices to replace $a-1$ with $1-a$ where $a-1$ appears).

\subsection{How to decrease $N_-$ using  Riemann zeta function.}
Typically, $N_->>N_+$.
The following  simple trick allows one to decrease $N_-$, in the case of the cpdf especially.
In the case of pdf,  for a chosen $m$, we precalculate $C^0_f$, $C^\al_f$,
$\ze_R(3m+1)$ and $\ze_R(m(2+\al)+1)$, where
$
\ze_R(s)=\sum_{j=1}^{+\infty} j^{-s}$  
is Riemann zeta function 
(we add the subscript $R$ to distinguish Riemann zeta function from
the mesh size $\ze$). Next, we define
\beqast
\La^0_-&=&\left(\frac{|C^0_f|(|x'|+1)}{4^m\eps\pi (3m+2)}\right)^{1/(3m+2)},\
\La^\al_1=\left(\frac{|C^\al_f|\al}{4^\al\eps\pi(m(2+\al)+2)}\right)^{1/((2+\al)m+2)},
\eqast
then set $\La_-=\max\{\La^0_-,\La^\al_-\}$, $N_-=\mathrm{ceil}\,(\La_- /\ze)$, 
 and calculate
\beqast
p(x')&=&p(0)+\frac{\ze}{\pi}\Re \sum_{j=-N_-}^{N_+}f(j\ze)+C^0_f\frac{\ze^{-3m}}{\pi}\left(\ze_R(3m+1)-\sum_{j=1}^{N_-}j^{-3m-1}\right)
\\
&& +C^\al_f\frac{\ze^{-(2+\al)m}}{\pi}
\cdot\left(\ze_R((2+\al)m+1)-\sum_{j=1}^{N_-}j^{-(2+\al)m-1}\right).
\eqast
In the case of cpdf, if only one leading term is separated, we precalculate $\ze_R(m+1)$ and $C^0_{f1}$, define
\beqast
\La^1_-&=&\left(\frac{|C^1_{f1}|}{\eps\pi(m+2)}\right)^{1/(m+2)},\ 
\La^{1\al}_-=\left(\frac{|C^{1\al}_{f1}|}{2\eps\pi m}\right)^{1/(2m)},\
\La^\al_-=\left(\frac{|C^\al_{f1}|}{\eps\pi(m(\al+1)+1)}\right)^{1/(m(\al+1))},
\eqast
$\La_-=\max\{\La^1_-,\La^\al_-,,\La^{1\al}_-\}$, $N_-=\mathrm{ceil}\,(\La_-/\ze),$
and calculate
\beqast
F(x')&=&0.5+\frac{\varphi_0}{\pi\al}+\frac{\ze}{\pi}\Re \sum_{j=-N_-}^{N_+}f_1(j\ze)
+C^0_{f1}\frac{\ze^{-m}}{\pi}\left(\ze_R(m+1)-\sum_{j=1}^{N_-}j^{-m-1}\right).
\eqast
If 4 first terms of the asymptotics of the integrand are taken into account, then $N_-$ decreases significantly.
We use an approximate prescription
\[
\La_-=\left((x')^2+|\Cp|^2)/(\eps\pi 4^{2m}\tilde m )\right)^{1/\tilde m},
\]
and the formula
\beqast
F(x')&=&0.5+\frac{\varphi_0}{\pi\al}+\frac{\ze}{\pi}\Re \sum_{j=-N_-}^{N_+}f_1(j\ze)
+C^0_{f1}\frac{\ze^{-m}}{\pi}\left(\ze_R(m+1)-\sum_{j=1}^{N_-}j^{-m-1}\right)\\
&& +C^1_{f1}\frac{\ze^{-m-2}}{\pi}\left(\ze_R(m+3)-\sum_{j=1}^{N_-}j^{-m-3}\right)
 +C^{1\al}_{f1}\frac{\ze^{-2m}}{\pi}\left(\ze_R(2m+1)-\sum_{j=1}^{N_-}j^{-2m-1}\right)\\
&& +C^{\al}_{f1}\frac{\ze^{-m(1+\al)}}{\pi}
\left(\ze_R(m(1+\al)+1)-\sum_{j=1}^{N_-}j^{-m(1+\al)-1}\right).
\eqast
\begin{rem}{\rm
Note that the use of 4 terms and more is efficient only
 if $x'$ is not large in absolute value, hence, in the cases $(++)$ and $(--)$
(implying $\al\in (0,1)$). It follows that, in cases $(+-)$ and $(-+)$, when $|x'|$ is very large, it may be advantageous to reduce the case
$\al\in (1,2]$ to the case $\al\in [0.5, 1)$. This can be done similarly to the reduction of the case $\al<1$ to the case $\al>1$, which we
consider below.}
\end{rem}

\subsection{Reduction of case $\al<1$ to case $\al>1$}\label{exoticalle1}
If $\al<1$ is small, then the integrands decay slowly, and it is advantageous to reduce the calculations to the case $\al>1$.
%In \eq{pdfSL1/2} and \eq{F}, we change the variables to reduce to similar integrals with the new $\tilde\al=1/\al>1$,
%and then apply the straightforward modification of the scheme in Section \ref{exoticalge1}. The change of variables
%is needed when the most efficient conic simplified rule faces difficulties; these cases  are

%$(+)$ $x'>0, \be>0$ (hence, $\varphi_0\in [-\pi\al/2,0]$), and

%$(-)$ $x'<0, \be<0$ (hence, $\varphi_0\in [0,\pi\al/2]$).

\paragraph{Case $(++)$, PDF.} In \eq{pdfSL1/2}, we change the variable $\Cp\xi^\al=-i\xi_1$, where $\xi_1$ runs over $\bR_+$
(for justification, we deform the line of integration to $-i\bR_+$), equivalently,\\ $|\Cp|e^{i\varphi_0}\xi^\al=e^{-i\pi/2}\xi_1$,
and, finally,
\bbe\label{xixi1}
\xi=|\Cp|^{-1/\al}e^{-i(\pi/2+\varphi_0)/\al}\xi_1^{1/\al}.
\ee
The deformation and change of variables can be justified if
\[
\Re(ix'(\xi(\xi_1)))=x'|\Cp|^{-1/\al}\xi_1^{1/\al}\Re e^{i(\pi/2-(\pi/2+\varphi_0)/\al)}
\]
tends to $+\infty$ as $\xi_1\to+\infty$, equivalently, $\pi/2>\pi/2-(\pi/2+\varphi_0)/\al>-\pi/2$. 
%Since$\varphi_0\in [-\pi\al/2,0]$, 
A sufficient condition (valid for arbitrary $\varphi_0\in [-\pi\al/2,0]$) 
is $\pi/2-\pi/(2\al)>-\pi/2\Leftrightarrow 1-1/\al>-1 \Leftrightarrow \al>1/2$. A necessary and sufficient condition
is $\varphi_0<\pi(\al-0.5)$; this condition implies that $\al>1/3$.

Assuming that $\varphi_0<\pi(\al-0.5)$, we calculate
\bbe\label{derxixi1}
\frac{d\xi}{d\xi_1}=\frac{1}{\al|\Cp|^{1/\al}}e^{-i(\pi/2+\varphi_0)/\al}\xi_1^{1/\al-1},
\ee
and, letting $\tilde\al=1/\al$, $\tilde\phi_0=\pi/2-(\pi/2+\varphi_0)/\al$, $\tilde C_+=x'|\Cp|^{-\tilde\al}e^{i\tilde\varphi_0}$,
obtain
\[
p(x')=\frac{\tilde\al}{\pi |\Cp|^{\tilde\al}}\Re \left(e^{-i(\pi/2+\varphi_0)\tilde\al}\int_{\bR_+}
e^{i\xi_1-\tilde C_+\xi_1^{\tilde \al}}\xi_1^{\tilde \al-1}d\xi_1\right).
\]
If $|\tilde C_+|$ is small, then it is advantageous to calculate the integral directly 
using the changes of variables for $x':=-1$: $\xi_1=i(-1+(1-ib\eta^m)^a)$ with $a>1$ s.t. 
$(\pi/2)(1-a)\tilde\al+\tilde\varphi_0\in (-\pi/2,\pi/2)$, and
\eq{yPchange}. To derive an approximate bound for the 
Hardy norm and the truncation parameter $\La_+$, we use $c_\infty=-\cos(a\pi/2)$. The asymptotics
of the integrand as $\xi_1\to 0$ is
\[
\xi_1^{\tilde \al-1}e^{-\tilde C_+\xi_1^{\tilde\al}+i\xi_1}=\xi_1^{\tilde \al-1}-(\tilde C_+)\xi_1^{2\tilde\al-1}+i\xi_1^{\tilde\al}+\cdots
\]
Deriving the asymptotics for $\xi_1(\eta(y))$ as $y\to-\infty$, and substituting into the above formula, we
can obtain the leading terms for the truncated part of the infinite sum in the neighborhood of $y=-\infty$ and
significantly decrease $N_-$. The details are left to the reader.

If $|\tilde C_+|$ is not small, then it may be advantageous to make the following reduction. 
Since
\[
\Re  \left(e^{-i(\pi/2+\varphi_0)\tilde\al}\int_{\bR_+}
e^{-\tilde C_+\xi_1^{\tilde \al}}\xi_1^{\tilde \al-1}d\xi_1\right)
=0,
\]
we have
\beqast
p(x')&=&\frac{\tilde\al}{\pi|\Cp|^{\tilde\al}}\tilde p(-1),
\eqast
where
\beqast
\tilde p(-1)&=&\Re \left(e^{-i(\pi/2+\varphi_0)\tilde\al}\int_{\bR_+}
\left(e^{i\xi}-1\right)e^{-\tilde C_+\xi^{\tilde \al}}\xi^{\tilde \al-1}d\xi\right).
\eqast

\paragraph{Case $(--)$, PDF.} We change the variable $\Cp\xi^\al=i\xi_1$, where $\xi_1$ runs over $\bR_+$
(for justification, we deform the line of integration to $i\bR_+$), equivalently, $|\Cp|e^{i\varphi_0}\xi^\al=e^{i\pi/2}\xi_1$,
and, finally,
\bbe\label{xixi1m}
\xi=|\Cp|^{-1/\al}e^{i(\pi/2-\varphi_0)/\al}\xi_1^{1/\al}.
\ee
The deformation and change of variables can be justified if
\[
\Re(ix'(\xi(\xi_1)))=-x'|\Cp|^{-1/\al}\xi_1^{1/\al}\Re e^{i(-\pi/2+(\pi/2-\varphi_0)/\al)}
\]
tends to $+\infty$ as $\xi_1\to+\infty$, equivalently, $-\pi/2<-\pi/2+(\pi/2-\varphi_0)/\al<\pi/2$. Since
$\varphi_0\in [0,\pi\al/2]$, a necessary and sufficient condition is $-\varphi_0<\pi(\al-0.5)$; this condition implies that $\al>1/3$.

Assuming that $-\varphi_0<\pi(\al-0.5)$, we calculate
\bbe\label{derxixi1m}
\frac{d\xi}{d\xi_1}=\frac{1}{\al|\Cp|^{1/\al}}e^{i(\pi/2-\varphi_0)/\al}\xi_1^{1/\al-1},
\ee
and, letting $\tilde\al=1/\al$, $\tilde\phi_0=-\pi/2+(\pi/2-\varphi_0)/\al$, 
$\tilde C_+=-x'|\Cp|^{-\tilde\al}e^{i\tilde\varphi_0}$,
obtain
\[
p(x')=\frac{\tilde\al}{\pi |\Cp|^{\tilde\al}}\Re \left(e^{i(\pi/2-\varphi_0)\tilde\al}\int_{\bR_+}
e^{-i\xi_1-\tilde C_+\xi_1^{\tilde \al}}\xi_1^{\tilde \al-1}d\xi_1\right).
\]
If $|\tilde C_+|$ is small, then it is advantageous to calculate the integral directly 
using the changes of variables for $x':=1$: $\xi_1=i(1-(1+ib\eta^m)^a)$ with $a>1$ s.t. 
$(\pi/2)(a-1)\tilde\al+\tilde\varphi_0\in (-\pi/2,\pi/2)$, and
\eq{yPchange}. 

If $|\tilde C_+|$ is not small, then it may be advantageous to make the following reduction. 
Since
\[
\Re  \left(e^{i(\pi/2-\varphi_0)\tilde\al}\int_{\bR_+}
e^{-\tilde C_+\xi_1^{\tilde \al}}\xi_1^{\tilde \al-1}d\xi_1\right)=0,
\]
we have
\beqast
p(x')&=&\frac{\tilde\al}{\pi|\Cp|^{\tilde\al}}\tilde p(1),
\eqast
where
\beqast
\tilde p(1)&=&\Re \left(e^{i(\pi/2-\varphi_0)\tilde\al}\int_{\bR_+}
\left(e^{-i\xi_1}-1\right)e^{-\tilde C_+\xi^{\tilde \al}}\xi^{\tilde \al-1}d\xi\right).
\eqast
We can calculate $\tilde p(\pm1)$ using the same changes of variables as in the case $\al>1$, $x'=\pm 1$,
with $\tilde\al$ in place of $\al$.
The choice of parameters modifies in the trivial manner; instead of $\eps$, 
$\tilde\eps=\eps\al|\Cp|^{\tilde \al}$ must be used. 

\paragraph{Case $(++)$ , CPDF.} In the integral on the RHS of \eq{PF}, we make the change of variables \eq{xixi1} and use \eq{derxixi1} and the same
notation $\tilde\al$, $\tilde\varphi_0$, $\tilde C_+$ as in Case (++) for pdf:
\[
\int_{\bR_+}\frac{e^{-ix'\xi}-1}{-i\xi}e^{-\Cp \xi^\al}d\xi=\tilde\al\int_{\bR_+}\frac{e^{-\tilde C_+\xi_1^{\tilde\al}}-1}{-i\xi_1}e^{i\xi_1}d\xi_1.
\]
We calculate the integral using the changes of variables for $x':=-1$: $\xi_1=i(-1+(1-ib\eta^m)^a)$ and
\eq{yPchange}. To derive an approximate bound for the 
Hardy norm and the truncation parameter $\La_+$, we use $c_\infty=-\cos(a\pi/2)$. The asymptotics
of the integrand as $\xi_1\to 0$ is
\[
\frac{e^{-\tilde C_+\xi_1^{\tilde\al}}-1}{-i\xi_1}e^{i\xi_1}=(\tilde C_+/i)\xi_1^{\tilde\al-1}+\tilde C_+\xi_1^{\tilde\al}+\cdots
\]
Deriving the asymptotics for $\xi_1(\eta(y))$ as $y\to-\infty$, and substituting into the above formula, we
can obtain the leading terms for the truncated part of the infinite sum in the neighborhood of $y=-\infty$, and
significantly decrease $N_-$. The details are left to the reader.

\paragraph{Case $(--)$, CPDF.} We make the change of variables \eq{xixi1m}-\eq{derxixi1m} and use the same
notation $\tilde\al$, $\tilde\varphi_0$, $\tilde C_+$ as in Case $(-)$ for pdf. 
We calculate the integral using the changes of variables for $x':=1$: $\xi_1=i(1-(1+ib\eta^m)^a)$ and
\eq{yPchange}. After that, we continue similarly to the case $(++)$.
 
 \section{Sub-polynomial acceleration}\label{s:SubP}
 
  \subsection{One-sided stable L\'evy distributions}\label{hyperOne-sided_subP}
  If $x'$ is large in absolute value, then the sinh-acceleration and polynomial acceleration can be inefficient
 due to the presence of the factor $e^{\om_1x'}$. In order that this factor be neither
 extremely small nor extremely large, $\om_1$ must be very small in absolute value.
 But then the strip of analyticity that can be used to derive the error bound for the Hardy norm is very narrow,
 hence, the mesh size is very small and the number of terms in the simplified trapezoid rule is very large.
 To tackle this difficulty, we use the change of variables of the form
 $
 \xi=y \ln^m(a^2+y^2)$, where $a>0$ and $m\ge 1$, and integrate along an appropriate line
 $\{\Im y=\om\}$. The idea is as follows. As $y\to\pm\infty$, the leading term of the asymptotics of the real part of 
 \beqast
 &&-ix'(y+i\om)\ln^m(a^2+(y+i\om)^2)\\
 &=& -ix'(y+i\om)(2\ln (y+i\om)+\ln(1+2i\om/y+O(y^{-2}))^m\\
 &=&(2\ln y)^m(-x')(y+i\om)\left(1+\frac{i\om}{y\ln y}+O(y^{-2})\right)^m\\
 &=& (2\ln y)^m(-ix' y+x'\om(1+m/\ln y)+O(y^{-1}))
\eqast
 is $x'\om (2\ln y)^m$. Hence, if we choose the strip of analyticity $S_{(\om-d, \om+d)}$ of the integrand
in the $y$-coordinate so that  $x'(\om\pm d)<0$, then the integrand is uniformly bounded by a small or moderately large
constant and decays fast as $y\to+\infty$ in the strip.
The logarithm of the integrand tends to $-\infty$ as $\exp[-\psi^0_{st}(|y|\ln^m|y|)]$, hence, much faster than  the
integrand in the $\xi$-coordinate, and the number of terms in the simplified trapezoid rule decreases significantly.
In the context of KoBoL and other L\'evy processes with exponentially decaying tails, similar changes of variables
were used in \cite{iFT}. The change is useful when the simplified conic trapezoid rule can be applied with a very narrow
cone only, and the strip of analyticity in the $y$-coordinate $\xi=e^{\i\om+y}$ is too narrow, 
hence, $\ze$ too small and the number of terms
 too large.

\subsection{General stable L\'evy distributions of index $\al=1$}\label{exotical1}
We have
\beqa\label{P1}
p(x')&=&\frac{\sg}{\sg^2+(x')^2}+\frac{1}{\pi}\Re\int_0^{+\infty} f_0(\xi)d\xi,\\\label{F1}
F(x')&=&\frac{1}{2}+\frac{\arctan(x')/\sg}{\pi}-\frac{1}{\pi}\Im\int_0^{+\infty}\frac{f_0(\xi)}{\xi}d\xi,
\eqa
where 
$f_0(\xi)=e^{-(ix'+\sg)\xi}\left(e^{-i\frac{2\sg\be}{\pi}\xi\ln\xi}-1\right).
$
In  integrals \eq{P1}-\eq{F1}, we change the variables
\[
\xi=\eta\ln^m\left(1+\eta^2/(2d_0)^2\right),\ \eta=0.5(y+(y^2+(2d_0)^2)^{1/2},
\]
where $d_0\ge 1$, $m\ge 1$. The derivatives are
\beqast
\frac{d\eta}{dy}&=&0.5(1+y(y^2+(2d_0)^2)^{-1/2})\\
\frac{d\xi}{d\eta}&=&\ln^{m-1}\left(1+\frac{\eta^2}{(2d_0)^2}\right)\left[\ln\left(1+\frac{\eta^2}{(2d_0)^2}\right)+\frac{2m\eta^2}{(2d_0)^2+\eta^2}\right],
\eqast
and the simplified trapezoid rule for pdf is of the form
\beqast\label{simplpdf}
\frac{1}{\pi}\Re\int_0^{+\infty} f_0(\xi)d\xi&=&\frac{\ze}{\pi}\Re\sum_{-N_-}^{N_+} f_0(\xi(\eta(y_j))\frac{d\xi}{dy}(y_j),\\
%\frac{1}{\pi}\Im\int_0^{+\infty} \frac{f_0(\xi)}{\xi}d\xi&=&
%\frac{\ze}{\pi}\Re\sum_{-N_-}^{N_+} \frac{f_0(\xi(\eta(y_j))}{\xi(\eta(y_j)}\frac{d\xi}{d\eta}(\eta(y_j))\frac{d\eta}{dy}(y_j),\\
\eqast
where $\frac{d\xi}{dy}(y_j)=\frac{d\xi}{d\eta}(\eta(y_j))\frac{d\eta}{dy}(y_j)$, $y_j=i\om+j\ze$.
Clearly, the rate of decay of the integrand increases with $m$ and the strip of analyticity
of the integrand widens as $d_0$ increases.  Hence, large $m$ and $d_0$ decrease the number
of terms in the simplified trapezoid rule (given the error tolerance). However, approximate bounds
for the Hardy norm, hence, for the discretization error, and bounds for the truncation errors, that we derive,
become too inaccurate for large $m$ and $d_0$. In numerical experiments, we observed
that the choice of $m\ge 2.5$ is unsafe for relatively large $|x'|$, and the choice of $d_0\ge 2.5$
can also be unsafe. Typically, the choice $m=2$ and $d=2$ and approximate recommendations that we derive are safe.

\paragraph{Choice of the strip of analyticity and $\ze$}%\label{exoticAl1_choice_strip}
\begin{lem}\label{lem:exoticAl1_choice_strip}
Functions $y\mapsto \eta(y)$ and $y\mapsto \xi(\eta(y))$ are analytic in the strip $S_{(-2d_0,2d_0)}$.
\end{lem}
\begin{proof} Rescaling $y\mapsto 2d_0$ reduces to the case $2d_0=1$. 
The analyticity of $\eta(y)$ is trivial, and the  analyticity  of $\xi(\eta(y))$
follows from
%\bbe\label{analxi}
\[
1+0.25(y+(1+y^2)^{1/2}))^2\not\in (-\infty,0],\ \Im y\in (-1,1).
\]
%\ee
For the  proof, it suffices to consider $y$ in the right half-plane. But then $\mathrm{arg} (y+(1+y^2)^{1/2})\in (-\pi/2, \pi/2)$,
and $\mathrm{arg} (y+(1+y^2)^{1/2})^2\in (-\pi, \pi)$.
\end{proof}
\begin{lem}\label{hyperosc1}
Let $x'<0$ (resp., $x'>0$). Then, for any $\om\in (0, 2d_0)$ (resp., $\om\in (-2d_0, 0)$),
$\Re (-ix'\xi(y+i\om))\to -\infty$ as $y\to+\infty$.
\end{lem}
\begin{proof} We have
\beqast
\eta(y+i\om)&=&\frac{1}{2}\left(y+i\om+(y+i\om)\left(1+\frac{(2d_0)^2}{(y+i\om)^2}\right)^{1/2}\right)\\
&=&y+i\om+2d_0^2(y+i\om)^{-1}+O(y^{-3}),
\eqast
therefore, 
\beqast
\xi(\eta(y+i\om))&=&(y+i\om+O(y^{-1}))\left[\ln((y+i\om)^2)+O(y^{-2})\right]^m\\
&\sim &y(2\ln y)^m(1-\ln (2d_0)/\ln y+i\om(1+m/\ln y)/y)
\eqast
and 
$
\Re(-ix'\xi)\sim (2\ln y)^mx'\om(1+m/\ln y).
$
\end{proof}
\begin{lem}\label{hyperosc2}  For any $\om\in (-2d_0, 2d_0)$, as $y\to+\infty$,
\beqast
&&\Re \left(-\sg\xi(\eta(y+i\om))\left(1+\frac{2\be}{\pi}\ln \xi(\eta(y+i\om))\right)\right)
\sim -\sg y(2\ln y)^m\to -\infty.
\eqast
\end{lem}
\begin{proof} We have $-\xi\sim y(2\ln y)^m$, and 
\[
\Re(-i(2\sg\be/\pi)\xi\ln\xi)\sim (2\ln y)^{m+1}2\om \sg\be/\pi.
\]
\end{proof}
We set $\om=-d_0\mathrm{sign}x'$, choose $k_d\in (0,1)$, set $d=k_dd_0$,
and use the line of integration $\{\Im\xi=\om\}$ and the strip $S_{(\om-d,\om+d)}$ to
derive the error bound for the infinite trapezoid rule and recommendation for the choice
of $\ze$. In the case of pdf (resp., cpdf), the integrand is
\[
f(y)=f_0(\xi(\eta(y))\frac{d\xi}{d\eta}(\eta(y))\frac{d\eta}{dy}(y)
\]
(resp., $f_1(y):=f(y)/\xi(\eta(y))$). The analysis of the proofs of Lemmas \ref{hyperosc1}-\ref{hyperosc2}
 shows that both 
 functions are uniformly bounded by a small or moderately large constant and uniformly decay
 as $\Re y\to+\infty$ and $y$ remains in the strip.

As a simple bound for the Hardy norm, we use
$H=|f(i(\om-d)|+|f(i(\om+d)|+1/\sg$ in the case of pdf; in the case of cpdf, we replace $f$ with $f_1$. 
Then we set $\ze=2\pi d/\ln(10H/\eps)$, where $\eps$ is the error tolerance. The bound can be easily improved,
and larger $\ze$ used.

\paragraph{Choice of $\La_+$ and $N_+$.}
As $\Re y\to +\infty$, $(d\eta/dy)(y+i\om)\sim 1$, and $(d\xi/dy)(y+i\om)\sim (2\ln y)^m$. Hence, 
we start with a simple approximate equation 
\[
\frac{1}{\pi}\int_{\La_+}^{+\infty} e^{-\sg y(2\ln y)^m-|x'\om|(2\ln y)^m}(2\ln y)^m dy=\eps,\]
which we replace with a simpler one
%\[
%\frac{1}{\pi}\int_{\La_++|x'\om|}^{+\infty} e^{-\sg y(2\ln y)^m}(2\ln y)^m dy=\eps,\]
%and then with an approximation
\[
\frac{1}{\sg\pi}\int_{\La_+}^{+\infty} e^{-\sg y(2\ln y)^m}d(\sg y(2\ln y)^m)=\eps.\]
Let  $Y_+=\ln\La_+$. We have the equation
\[
e^{Y_+}Y_+^m=-\ln(\eps\sg\pi)/(\sg 2^m),
\]
which can be solved using the Newton method. When $Y_+$ is calculated with a moderate precision,
we set $\La_+=\exp Y_+$, $N_+=\mathrm{ceil}\, (\La_+/\ze)$. Note that
this prescription may lead to a serious overkill, in the case of cpdf  and large $|x'|$ especially.
\paragraph{Choice of $\La_-$ and $N_-$}
We calculate the asymptotics of $\eta(y+i\om)$, $\xi(\eta(y+i\om))$, their derivatives and 
$f(y+i\om)$ as $y\to-\infty$:
\beqast
\eta(y+i\om)&=&\frac{1}{2}(y+i\om+(-y-i\om)\left(1+\frac{(2d_0)^2}{(-y-i\om)^2}\right)^{1/2})\\
&=&d_0^2(-y)^{-1}(1+i\om/(-y)+O(y^{-2})),\\
\xi(\eta)&=&\eta(\eta^2/(2d_0)^2+O(\eta^4))^m\\
&=& \eta^{2m+1}(2d_0)^{-2m}(1+O(\eta^2)),
\eqast
then
\beqast\nonumber
\xi(\eta(y+i\om))&=&\left(\frac{d_0^2}{-y}\right)^{2m+1}\frac{1}{(2d_0)^{2m}}
\left(1+\frac{i\om(2m+1)}{-y}+O(y^{-2})\right)\\\nonumber
&=&\frac{d_0^{2(m+1)}}{(-y)^{2m+1}}\left(1+\frac{i\om(2m+1)}{-y}+O\left(\frac{1}{(-y)^{2}}\right)\right),\\
\ln(\xi(\eta(y+i\om)))
&=&-(2m+1)\ln(-y)+\frac{i\om(2m+1)}{-y}+g_0(y)+ig_{-2}(y),
\eqast
where $g_0(y)=O(1)$ and $g_{-2}(y)=O(y^{-2})$ are real-valued functions.
We conclude that 
\beqast
-i(\xi\ln\xi)(\eta(y+i\om))
&=&(2m+1)\frac{d_0^{2(m+1)}}{(-y)^{2m+1}}\ln(-y)\left[i-\frac{\om(2m+1)}{-y}+g(y)\right],
\eqast
where $g(y)=ig_{-1}(y)+g_{-2}(y)$, and $g_{-1}=O(y^{-1})$, $g_{-2}(y)=O(y^{-2})$ are real-valued functions,
and 
\beqast
\frac{d\xi}{dy}(y+i\om)
&=&d_0^{2(m+1)}\left[\frac{2m+1}{(-y)^{2(m+1)}}+\frac{i\om(2m+1)(2m+2)}{(-y)^{2m+3}}+\cdots\right].
\eqast
Finally, 
\beqast
f(y)&=&e^{-(ix'+\sg)\xi}\left(e^{-i\frac{2\sg\be}{\pi}\xi\ln\xi}-1\right)\frac{d\xi}{dy}(y+i\om)\\
&=&\frac{2\sg\be}{\pi}d_0^{4(m+1)}(2m+1)^2(-y)^{-4m-3}\ln(-y)
\left[i-\frac{\om(4m+3)}{-y}+\cdots\right],
\eqast
and, therefore,
\bbe\label{fas}
\Re f(y)=C_f(-y)^{-4(m+1)}\ln(-y)(1+O(y^{-1})),
\ee
where
\[
C_f=-\om\frac{2\sg\be}{\pi}d_0^{4(m+1)}(2m+1)^2(4m+3).
\]
Using the asymptotic formula \eq{fas}, we find the truncation
parameter $\La_-$ in the simplified trapezoid rule for pdf as a solution of
the inequality
\[
\int_{\La_-}^{+\infty} y^{-4m-4}\ln y dy\le \eps_1,
\]
where 
\[
\eps_1=\frac{\eps\pi}{2\sg\be\om}d_0^{-4(m+1)}(2m+1)^{-2}(4m+3)^{-1}.
\]
Integrating by parts, we see that, as an approximation, we may use
the solution of the equation
\[
\La_-e^{-(4m+3)\La_-}=\eps_1(4m+3),
\]
which is equivalent to
\[
\La_--\ln\La_-/(4m+3)+\ln(\eps_1(4m+3))/(4m+3)=0.
\]
An approximate solution can be easily found using the Newton method. 
In the case of the cpdf, we similarly derive
\bbe\label{f1as}
\Im\frac{f(y)}{\xi(\eta(y+i\om))}\sim C_{f1}(-y)^{-2(m+1)}\ln(-y),
\ee
where $C_{f1}=(2\sg\be/\pi)d_0^{2(m+1)}(2m+1)^2$, and find $\La_-$ as
an approximate solution of 
\[
\La_--\ln\La_-/(2m)+\ln(2m\eps_1)/(2m)=0,
\]
where 
$
\eps_1=\eps\pi(2\sg\be\om)^{-1}d_0^{-2(m+1)}(2m+1)^{-2}.
$
Finally, we set
$N_-=\mathrm{ceil}\,(\La_-/\ze)$.

\paragraph{How to decrease $N_-$ using \eq{fas}, \eq{f1as}  and Riemann zeta function.}
Typically, $N_->>N_+$.
The following  simple trick allows one to decrease $N_-$, in the case of the cpdf especially. 
For $m$ that is used in our method for cpdf, we precalculate $\ze_R(2(m+1))$ and $\ze'_R(2(m+1))$, where
$
\ze_R(s)=\sum_{j=1}^{+\infty} j^{-s}$ and $
\ze'_R(s)=-\sum_{j=1}^{+\infty} j^{-s}\ln j$
are Riemann zeta function and its derivative.
Due to \eq{f1as}, the rate of decay of
\[
f_2(y_j):=f_1(y_j)-C_{f1}\ze^{-2(m+1)}(-j)^{-2(m+1)}\ln(-j\ze)\]
as $j\to-\infty$ is larger than the rate of decay of $f_1(y_j)$. Hence, we may use an approximate
equation 
\[
\La_--\ln\La_-/(2m+1)+\ln(\eps_1(2m+1))/(2m+1)=0
\]
to define $\La_-$ and then $N_-$, and the formula for the cpdf becomes
\beqa\label{cpdfAl1Riemann}
F(x')&=&\frac{1}{2}+\frac{\arctan(x'/\sg)}{\pi}-\frac{\ze}{\pi}\Im\sum_{j=-N_-}^{N_+} f_1(y_j)
-\frac{ C_{f1}}{\ze^{2m+1}\pi}(-\ze'_R(2(m+1))\\\nonumber
&&+\ze_R(2(m+1))\ln\ze-\sum_{j=1}^{N_-}j^{-2(m+1)}\ln(j\ze))
\eqa
\begin{rem} {\rm One can decrease $N_-$ further still deriving several terms of the asymptotics
of $f(y)$ and $f_1(y)$ as $y\to-\infty$ and generalizations of \eq{cpdfAl1Riemann}
with several correction terms expressible in terms of Riemann zeta function and its derivatives.}
\end{rem}

\subsection{Stable L\'evy distributions of index $\al\neq 1$}\label{exoticalneq1}
In integrals \eq{Pp} and \eq{PF}, we make the same change of variables as in Section \ref{exotical1},
and choose the same strip of analyticity $S_{(\om-d,\om+d)}$. 
We omit the details because in our numerical experiments, the polynomial acceleration
required less time than the sub-polynomial one.

\section{Numerical examples}\label{numer}
\subsection{General remarks}
The calculations in the paper
were performed in MATLAB 2017b-academic use, on a MacPro with a
2.8 GHz Intel Core i7 and 16 GB 2133 MHz LPDDR3 RAM. 
 Errors $Err$ are differences between $p(x)$ (or $F(x)$) calculated for the parameters of the scheme indicated in the tables and the benchmark; the latter 
satisfies the error tolerance  smaller than the errors shown.
In all cases, $\sg=0.001$, and $\al, \be, x'$ vary. Thus, $x'\approx -5000$ in Table 1 corresponds to $x'\approx 5\cdot 10^6$ with the normalization $\sg=1$. The CPU time is in microsec.
In all cases, the location parameter $\mu$ is 0 in the $S_0$ Nolan's parametrization,
 hence, $x$ in the $S_0$ parametrization and $x'$ are related as $x'=x+\sg\be\tan(\pi\al/2)$. 
 
 In order that the reader could compare the results in the paper to the results produced by 
 John Nolan's program stable.exe (N) (available at http:\\ //fs2.american.edu/jpnolan/www/stable/stable.html),\\
  in all the tables but two we calculate pdf and cpdf as functions of $x$. As we explained in the main
  body of the text, many computational difficulties arize in a small neighborhood of $x'=0$;
  to illustrate these difficulties, in Table 4 and 5, we show pdf and cpdf as functions of $x'$.

The tables illustrate the following practical implications of the theoretical analysis
of errors of the methods of the paper: if $\al\in (0,1)$, the serious problems 
 are in a right vicinity of $x'=0$ if $\be>0$ and a left one if $\be<0$; the vicinity becomes rather large if $\al<1$ is close to 1. 
If $\al>1$ is close to 1, then 
the number of terms becomes very large  if $\be$ and $x'$ are of the opposite sign, and $|x'|$ is large.
 The same problem appears when $\al=1$. Thus, the bad region is where $x'\be\tan(\pi\al/2)<0$, and $|x'|$ is small if $\al\in (0,1)$ and large if $\al\in [1, 2)$.
As a side remark: this observation explains why it may be more natural to use $-\be$ instead of $\be$ for $\al\in (0,1)$; then, in all cases,
the bad region from the point of view
of difficulties for the numerical realizations  is a subset of a region where $x'\be<0$. In Tables 7 and 8, we show that accurate calculations in this region are difficult not only for
the methods of the present paper but for popular methods as well. In some cases, the errors of the popular methods are sizable. For  pdf, the errors are less pronounced, and, typically, if $\al$ is not small, then, in the regions that are good for our methods,
the errors of (N) are of the same order of magnitude as of the methods of the current paper. If $\al$ is close to 1, then
the methods of the present paper are more accurate. 

We compare the performance of several realizations of the simplified conic trapezoid rule:
\vskip0.1cm
\noindent
$ConicE(k_\ga, N_{sc}, ord)$: simplified conic trapezoid rule with the universal choice e.-f. of the cone of analyticity. The $k_\ga$ means that,
in the general procedure for the parameter choice,
$\ga^\pm_0$ replaced with $k_\ga\ga^\pm_0$, where $k_\ga\in\{1,1.02,1.05,1.1\}$, the scaling  $\xi\mapsto 10^{N_{sc}}\xi$ is made, and 
the Taylor expansion of order {\em ord} is used in the truncated part of the left tail in the simplified conic part.
\vskip0.1cm
\noindent
$ConicC(k_\ga, N_{sc}, ord)$: simplified conic trapezoid rule with the universal choice c.-d. of the cone of analyticity; the
meaning of $k_\ga, N_{sc}, ord$ is the same.
\vskip0.1cm
\noindent
$ConicA(N_{sc}, ord)$: simplified conic trapezoid rule with the universal choice $\om=\pm\pi/2$ of the line of integration;
the meaning of $N_{sc}, ord$ is the same.

\vskip0.1cm
\noindent
In the cases $\al=1$ and $\al$ close to 1, we also include the results obtained with the subpolynomial method SubP and polynomial P, in the bad region where
the simplified conic trapezoid rule is very inefficient. 

\vskip0.1cm
In the tables, $\ze$ is the mesh size,  $-N_-$ and $N_+$ are the bounds in the sum of the simplified
conic trapezoid rule; the total number of terms is $N_++N_-$. Due to the rescaling, in some cases,
$N_+$ is negative. Unless otherwise stated, $\ze, N_\pm$ 
 are chosen using the general prescriptions in the paper. In many cases, the total number of terms and the CPU time
can be made smaller, sometimes, significantly.

In the tables, pdf and cpdf values used as the benchmark are obtained using one of the realizations
of the simplified conical trapezoid rule. The results obtained with these realizations agree very well (typically, the absolute differences are of the order of
$10^{-15}-10^{-18}$ for pdf and $10^{-12}-10^{-15}$ for cpdf)
 with the ones obtained with the other realizations,
hence, when different contours of integration and different additional sources of errors.
  In the tables, typically, $\ze$ is in the range 0.9-1.1 times the recommended
by general prescriptions, $\La_-$ is 1.2-1.3 times larger,
and $\La_+$ is as recommended, for the error tolerance $10^{-15}$.

\subsection{Case of $\al\in (0,1)$}
Tables \ref{tab:1}-\ref{tab:2}
demonstrate that if index $\al$ is small, tail calculations 
of pdf and cpdf are easy and can be made fast.
A neighborhood of 0 where accurate calculations require very large $N_\pm$ shrinks as $\al\to 0$
but, in this neighborhood, $N_\pm$ and CPU time may become very large.
The scaling can help  to alleviate this problem if $\al$ is not close to 1. 
%In particular, choosing a very large scaling
%parameter in the first example (Tables 1-2, $\al=0.15$, $\be=0.75$), we can use $ConicE(k_\ga, N_{sc}, ord)$ for all $x'$ and 
%$ConicA(N_{sc}, ord)$ if $x'$ is not close to 0, choosing appropriate $N_{sc}$. 
The error tolerance of the order of $10^{-15}$ and even less
can be satisfied with $100-300$ terms; the CPU time is 12-30 microseconds for the pdf; for cpdf, the CPU time is 2-3 
times larger. 

Typically, for $ConicA(N_{sc}, ord)$, the CPU time
is approximately 3 times less than the CPU time for $ConicE(k_\ga, N_{sc}, ord)$, because we can use especially simple formulas,
which require operations over reals only. $ConicC(k_\ga, N_{sc}, ord)$ uses the most conservative choice of the cone of analyticity,
hence, both $\ze$ and $\La_\pm=N_\pm\ze$ are larger; the CPU time is larger as well. 

\begin{table}
\caption{PDF and errors of $ConicA(Nsc, 2)$  with the meshes
recommended for the error tolerance e-15. $Nsc=-3, -1, 4, 0$ for $x\in [-5000,-100]$, $(-100,-5)$,
$[-0.0025, 0.0025]$, $[5,100]$, respectively.
$S_0$-parameters: $\al=0.15$, $\be=0.75$, $\sigma=0.001$, $\mu=0$.
Benchmark values are calculated using $ConicE(1.1, Nsc, 2)$   with  meshes longer and finer than recommended.
T: CPU time in microsec., the average over 100k runs. Mesh size $\ze$ is rounded. }
\label{tab:1}       % Give a unique label
\begin{tabular}{lllllll}
\hline\noalign{\smallskip}
$x$ & $p$ & $\ze$ & $N_-,N_+$ & Err & T \\
\noalign{\smallskip}\hline\noalign{\smallskip}
-5000 & 3.11318963730012e-7 & 0.13 &	192,27 & e-20 &	15\\
-3000 & 5.55907874099697e-7 & 0.13 &	194,33 &	e-20	& 14 \\
-1000 & 1.93023496327088e-6 & 0.13&	198,33 &	e-20&	14\\
-100 &2.59229551150544e-5 & 0.13 & 203,68 & e-20 & 15\\
-50 & 5.64483170567281e-5	&		0.13 &	204,24 &	e-18 &	15\\
-5 &	7.36841595407147e-4 &			0.13 &	211,52 &	e-18 &	18\\
-2.5e-3& 2.81289214828798 & 0.12 &	228,7 & e-15 & 17\\
-e-3 & 8.07337068614118	& 0.12 &	232,19 & e-14	&18\\
-e-4 & 581.201482282709	 & 0.12 &	237,38 &	e-15	&19\\
0 & 267.419034150846 &  0.12 &	236,33 &	e-15	&19\\
e-4&	173.7956347186	& 0.12 & 335,29 &	e-15	&19\\
e-3 &	41.3125849331846	& 0.12 & 231,15&	e-15	&16\\
2.5e-3 &	17.8476636093813	& 0.12 & 227,6 &	e-13	&15\\
5 & 5.263762423550393e-3	&		0.13&	211,20&	e-17	& 16\\ 
50 &4.010585652677472e-4	&		0.13 &	201,-1	& e-16& 13\\
100 & 1.83927301369793e-4	&		0.13 &	199,-7	& e-15 &13\\
\noalign{\smallskip}\hline
\end{tabular}
\end{table}

\begin{table}
\caption{CPDF and errors  of $ConicE(1.0, 4, 2)$ for $|x|\ge 0.01$ and $ConicE(1.0, 5, 2)$ for $|x|< 0.01$, with the meshes
2 times finer and longer than recommended for the error tolerance e-15.
$S_0$-parameters: $\al=0.15$, $\be=0.75$, $\sigma=0.001$, $\mu=0$.
Benchmark values are calculated using $ConicE(1.1, 4, 2)$ and $ConicE(1.1, 5, 2)$, respectively, with  meshes longer and finer than recommended.
T: CPU time in microsec., the average over 100k runs. Mesh size $\ze$ is rounded.}
\label{tab:2}       % Give a unique label
\begin{tabular}{lllllll}
\hline\noalign{\smallskip}
$x$ & $F$ & $\ze$ & $N_-,N_+$ & Err & T \\
\noalign{\smallskip}\hline\noalign{\smallskip}
-0.05 & 0.049667225184202	&	0.10 &	744,-24 &	e-10 &	73 \\ 
-0.01 & 0.059283790693812 &	0.10 &	702,-8	&    e-10 &	72\\
-e-3 & 0.0753320547491334 & 0.10 &	708,-1 &	e-10 &	74\\
-2e-4 & 0.0987472059268357 & 0.09 & 680,32 &e-10 & 73\\
-e-4 & 0.345499975862968	&	0.09 &	697,22 & e-10 &	73\\
0 &  0.383929540797937	&	0.10 &	703,14	&e-10 &	73 \\
e-4 & 0.405323108653059 & 0.10 &705,-10 & e-10 & 73\\
2e-4 &	0.420220377372895	&	0.10 &	706,7	& e-10 &	73\\
e-3 & 0.475668768917641 & 0.10 & 708,-5 & e-10 & 75\\
0.01 &		0.577540169623405	&	        0.10 &	702,-8	& e-10	&      72\\
0.05 &	0.646279180194864	&	0.10 &	704, -24 &	e-10 &	72\\
\noalign{\smallskip}\hline
\end{tabular}
\end{table}

\subsection{Case $\al\in (0,1)$, $\al$ very close to 1}
Table \ref{tab:3} demonstrates that even if $\al$ is very close to 1, namely, $\al=0.998$, and $\be=0.75$, that is,
the asymmetry is sizable, accurate and fast calculations in the tails are possible. 
%Closer to the origin, calculations
%are also not difficult because the scaling can be used to increase $|x'|$ as compared to $\sg^\al$. 
The real difficulties arise when $x'>0$ is small (Tables \ref{tab:4}-\ref{tab:5}). Since $\al$ is very close to 1,
the rescaling helps if we use the safest choice
of the cone but not other choices.  However, if the polynomial acceleration is used, then the number of terms decreases by two orders of magnitude,
and accuracy and speed of calculations increase as $x'\to 0$ (see Table \ref{tab:5b}).
In Table \ref{tab:6}, we compare the results with John Nolan's program stable.exe.
\begin{table}
\caption{PDF and errors of $ConicE(1.0, -2, 2)$ with the mesh
recommended for the error tolerance e-15.
 $S_0$-parameters: $\al=0.998$, $\be=0.75$, $\sigma=0.001$, $\mu=0$.
Benchmark values are calculated using $ConicE(1.1, -2, 2)$ with the mesh longer and finer than recommended.
T: CPU time in microsec., the average over 100k runs. Mesh size $\ze$ is rounded.}
\label{tab:3}       % Give a unique label
\begin{tabular}{lllllll}
\hline\noalign{\smallskip}
$x$ & $p$ & $\ze$ & $N_-,N_+$ & Err & T \\
\noalign{\smallskip}\hline\noalign{\smallskip}
-100 & 8.13536349845171e-9 &			0.19&	29,19 &	e-18	& 9.8 \\ 
-50 & 3.24934924707529e-8	&		0.17 &	29,23 &	e-18 &	9.6&\\
-25 & 1.297726494011055e-7	&		0.17 &	28,27 & e-20	& 9.9\\
-5 & 3.23031522416717e-6 &			0.18&	27,38 &	e-18&	11\\ 
5 & 2.26783179758502e-5	&		0.18&	27,38 &	e-20	& 12\\
25 & 9.09052669268316e-7 &			0.17 &	28,27 & e-20	& 10\\
50 & 2.27541207991646e-7	&		0.17&	29,23	& e-18	& 9.9 \\
100 & 5.69591734267896e-8	&		0.17&	29,19	&  e-17 &	9.4 \\
\noalign{\smallskip}\hline
\end{tabular}
\end{table}

\begin{table}
\caption{PDF and errors of $ConicE(1.0, 0, 2)$ with the mesh
recommended for the error tolerance e-15, for $x'$ close to 0 ($x\in (-0.2407,-0.2367)$).
 $S_0$-parameters: $\al=0.998$, $\be=0.75$, $\sigma=0.001$, $\mu=0$.
Benchmark values are calculated using $ConicE(1.1, 0, 2)$ with the mesh longer and finer than recommended.
T: CPU time in microsec., the average over 100000 runs. 
Mesh size $\ze$ and $N_\pm$ for $x'>0$ are rounded. }
\label{tab:4}       % Give a unique label
\begin{tabular}{lllllll}
\hline\noalign{\smallskip}
$x'$ & $p$ & $\ze$ & $N_-,N_+$ & Err & T \\
\noalign{\smallskip}\hline\noalign{\smallskip}
-2e-3 &1.36221356071656e-3 &		0.08 &	98,155 &	e-15& 42 \\ 
-1e-3 & 1.37350953475699e-3	&	0.09 &	98,155 &	e-14 & 42\\
-1e-4 & 1.38379659864617e-3	&	0.09 &	98,155 & e-14	& 40\\
1e-4 & 1.386099077465777e-3	&	e-4 &	110k,383k & e-11 &	29k\\
1e-3 & 1.39652771252394e-3 &	e-4 &	110k,383k	& e-9 & 30k\\
2e-3 &1.4082546893007e-3	 &	e-4 &	110k,383k & e-10&	29k\\
\noalign{\smallskip}\hline
\end{tabular}
\end{table}

\begin{table}
\caption{CPDF and errors of $ConicC(1.0, 3, 2)$ with the mesh
recommended for the error tolerance e-15, for $x'$ close to 0 ($x\in (-0.2407,-0.2367)$).
 $S_0$-parameters: $\al=0.998$, $\be=0.75$, $\sigma=0.001$, $\mu=0$.
Benchmark values are calculated using $ConicC(1.1, 3, 2)$ with the mesh longer and finer than recommended.
T: CPU time in microsec., the average over 10k runs. 
Mesh size $\ze$ and $N_\pm$ for $x'>0$ are rounded. }
\label{tab:5}       % Give a unique label
\begin{tabular}{lllllll}
\hline\noalign{\smallskip}
$x'$ & $F$ & $\ze$ & $N_-,N_+$ & Err & T \\
\noalign{\smallskip}\hline\noalign{\smallskip}
-2e-3 &0.33125085132895e-3 &		0.21 &	95,-9 &	e-18 &	24 \\
-1e-3 & 0.33261870120455e-3 &		0.21 &	95,-9 &	e-18&	16\\
-1e-4 & 0.33385948032056e-3	&	0.20 &	95,-9	& e-17	& 16\\
1e-4 & 0.33413646971679e-3	&	2.3e-4 &	81k,19k &	e-17 &	3.2k\\
1e-3 & 0.33538864255410e-3	&	2.4e-4 &	82k,18k &	e-15	& 3.2k\\
2e-3 & 0.33679102124245e-3	&	 	2.5e-4 &	82k,18k & e-16&	3.5k\\
\noalign{\smallskip}\hline
\end{tabular}
\end{table}

\begin{table}
\caption{CPDF and errors of polynomial acceleration with the parameters $(a,m,b)=(1,2,1)$ and the mesh
recommended for the error tolerance e-15, for $x'$ very close to 0.
 $S_0$-parameters: $\al=0.998$, $\be=0.75$, $\sigma=0.001$, $\mu=0$.
Benchmark values are calculated using several sets of parameters $(a,m,b)$ and fine and long grids; the
differences are smaller than e-12.
T: CPU time in microsec., the average over 10k runs. 
Mesh size $\ze$ is rounded. }
\label{tab:5b}       % Give a unique label
\begin{tabular}{lllllll}
\hline\noalign{\smallskip}
$x'$ & $F$ & $\ze$ & $N_-,N_+$ & Err & T \\
\noalign{\smallskip}\hline\noalign{\smallskip}
e-5 & 0.33401176744949e-3 & 0.07 & 91,2084	& e-12 &295\\ %old 0.334011767517955
5e-5 & 0.33406717919133e-3 & 0.07 & 158,2084 & e-13 &	304\\%0.0.334067179198411
e-4 & 0.33413646970775e-3 & 0.07 &	205,2084 & e-12 &	321\\%0.334136469709546
2e-4 & 0.33427513717746e-3 & 0.07 &	269,2084 & e-12 &	347\\%%0.334275137174523
e-3 & 0.33538864255737e-3 & 0.08 &	267,1823 &	e-12 &		327\\%0.335388642477815
2e-3 & 0.33679102125297e-3& 0.07 & 352,1823 &	e-12 &362\\% 0.336791021098592
\noalign{\smallskip}\hline
\end{tabular}
\end{table}

\begin{table}
\caption{CPDF and relative errors of John Nolan's program stable.exe (N) and $ConicE(1, 3, 2)$ (for $|x|\le 5$),
and $ConicE(1, 6, 2)$ (for $x\ge 5$) with the mesh 50\% finer and larger than
recommended for the error tolerance e-15, w.r.t. $\min\{F, 1-F\}$.
 $S_0$-parameters: $\al=0.998$, $\be=0.75$, $\sigma=0.001$, $\mu=0$.
Benchmark values are calculated using $ConicE(1, 3, 2)$ 
and $ConicE(1, 6, 2)$, respectively, with the meshes longer and finer than recommended.
$T_E$: CPU time in microsec., the average over 10k runs.  }
\label{tab:6}       % Give a unique label
\begin{tabular}{llllll}
\hline\noalign{\smallskip}
$x$ & $F$ &  $Err_E$  & $T_E$& $Err_N$ \\
\noalign{\smallskip}\hline\noalign{\smallskip}
-100	& 	8.15206374458673e-7 & e-13 &37 & -0.024\\
-50 & 1.62807802859660e-6 & e-13 & 35 & -0.022\\
-5	& 	1.61949951656763e-5 & e-14 & 36& -0.018\\
-2.5 &		3.23243097796957e-5  &e-14 & 47 &-0.017\\
-0.5	& 	1.60438900411786e-4& e-15 &40 & -0.013\\
-0.1		& 7.88201747983219e-4	& e-10 & 2.0k &-0.01\\
0 &		0.402108433490376	      & e-14 & 53 & 0.0035\\
0.1	& 	0.994257893316732 & e-12 & 45 & 0.035\\
0.5	& 	0.998864393911454	& e-11 & 44 & 0.019\\
2.5 & 0.999773085851662	& e-11 & 44 & 0.018\\
5	& 	0.999886458587786	& e-14 & 41 & 0.019\\
50 & 0.999988601171594 & e-10 & 39 & 0.023\\
100 & 0.999994292945519	& e-8 & 38 & 0.024\\
\noalign{\smallskip}\hline
\end{tabular}
\end{table}

\subsection{Case $\al=1$} To save space, we show the results for the cpdf only. 
In a certain sense, the case $\al=1$ is extreme.  As $x'<0$ decrease, the number of terms in 
the simplified conic trapezoid rule increases and becomes extremely large;
and the CPU time is measured in seconds; if $x'>0$, and not too small, it suffices to sum up several dozen of terms to satisfy the error tolerance of the order 
of E-15 and smaller; the CPU time is about 5 microseconds. Sub-polynomial acceleration requires
about 1 msec, and the number of terms and CPU time are of the same order of magnitude as $x'$ varies from -5 to -200.
See Tables \ref{tab:7}-\ref{tab:7b}.

\begin{table}
\caption{CPDF and relative errors (w.r.t. $\min\{F, 1-F\}$) of John Nolan's program stable.exe (N) and $ConicE$ with the mesh
recommended for the error tolerance e-15.
 $S_0$-parameters: $\al=1$, $\be=0.25$, $\sigma=0.001$, $\mu=0$.
Benchmark values are calculated using the general recommendation with the scale parameter $10^4$
 for $x'<-3$, and 1 for $x>-3$. 
 Errors of the benchmark: of the order of e-10-e-11 for $x'>-3$, and of the order of e-11-e-15 for $x'<-3$.
 $N_++N_-$ decreases as
$x$ increases: at $x=-1$, $\approx 110 mln$, at $x=-1$,
$\approx 50k$, at $x=-0.01$, $\approx 6k$, at $x=0.01$, $N_\pm=35$, and at $x=1000$, $N_+=42$ and $N_-=3$.
$T$: CPU time in microsec., the average over 1000 runs. 
}\label{tab:7}       % Give a unique label
\begin{tabular}{llllll}
\hline\noalign{\smallskip}
$x$ & $F$ &  $Err_E$  & $T$& $Err_N$ \\
\noalign{\smallskip}\hline\noalign{\smallskip}
-200 &		1.19365074579989e-6	& 3.6e-8 & 3.6mln & 4.2e-6\\
-80	&  	2.98409066904609e-6	& 9.1e-10 &1.4mln & 1.1e-5\\
-5	&	4.77341833448053e-5	& 2.1e-11 & 62k & 1.7e-4\\
 -1	&	2.38486189482856e-4	& -1.1e-10 & 11k & 7.5e-11\\
 -0.1	&	2.37147293339784e-3 & -1.4e-11 & 1.1k & 2.4e-12\\
  -0.01 &		0.0231302631073184	& 3.6e-11 &145 & 3.2e-13\\
  0	&	0.470104449706134	& 1.9e-14 &75 & 4.5e-13\\
0.01	 &	0.959213200641451 & 1.3e-13 &6.2 & 4.5e-14\\
1	&	0.999601701819188 & 1.0e-11 &5.1 & 8.5e-4\\
10 &	 	0.999960205698953	&  e-14 & 5.0 &8.5e-5\\
20	&	0.999980104131106	&		e-14	&	5.8  &	 4.3E-05\\
50	&	0.99999204198948	&	5.2e-9 & 4.5 & 1.4e-10\\
250	&	0.999998408438404	& e-14 & 4.7 & 5.4e-12\\
1000	&	0.999999602111794 & e-14 &  4.8 & 3.4e-13\\

\noalign{\smallskip}\hline
\end{tabular}
\end{table}

\begin{table}
\caption{CPDF and relative errors of sub-polynomial acceleration with a moderate number of terms.
Relative errors of the benchmark values are in the range $2e-14-2e-13$.
The parameters of the distribution are as in Table \ref{tab:7}.
CPU time  (the average over 10k runs) is in the range 0.9-1.2 msec }\label{tab:7b}
\begin{tabular}{llllll}
\hline\noalign{\smallskip}
$x$ & $F$ & $\ze$ & $N_-,N_+$ &  Err \\
\noalign{\smallskip}\hline\noalign{\smallskip}
-200 & 1.19365074579989e-6 & 0.13 &	3936,	74 &  -1.2e-7\\
-140 & 1.70520938216014e-6 & 0.15 &3478,	67 & -7.6e-8\\
-80 & 2.98409066904609e-6 & 0.17  &3140,	62 & -6.0e-8\\
-20 & 1.19357202286944e-5 &0.17 & 3118,	63 & 6.2e-13\\
\noalign{\smallskip}\hline
\end{tabular}
\end{table}

\subsection{Case $\al>1$} If $\al$ is not very close to 1 and $x'$ is not too large in absolute value, several hundred of terms of the safest realization suffice;
the CPU time is about 20-40 microseconds. See Tables \ref{tab:8} and \ref{tab:9}. The differences with (N) are of the same order of magnitude as the errors shown in the tables.

\begin{table}
\caption{PDF and errors of $ConicC(1,6, 2)$ with the mesh
recommended for the error tolerance e-15.
$S_0$-parameters: $\al=1.3$, $\be=0.25$, $\sigma=0.001$, $\mu=0$.
Benchmark values are calculated using $ConicC(1,1, 2)$ with  the mesh longer and finer than recommended.
T:  CPU time in microsec., the average over 10k runs. Mesh size $\ze$ is rounded.}
\label{tab:8}       % Give a unique label
\begin{tabular}{lllllll}
\hline\noalign{\smallskip}
$x$ & $p$ & $\ze$ & $N_-,N_+$ & Err & T \\
\noalign{\smallskip}\hline\noalign{\smallskip}
-125	& 	4.6979494046576e-10 & 0.044 &	307,180 &	2e-14 &	57\\
-25 &	1.90316902311684e-8	&	0.044 &	280,180 &	1e-14 &	53\\
-5	&	7.70985222878323e-7	&	0.044 &	252,180 &	-2e-13	& 48\\
-1	&	3.12191495821328e-5 &	0.044 &	225,180 &	-2e-14 &	52\\
-0.1	& 6.20796148088551e-3 &		0.044 &	186,180 &	2e-14	& 43 \\
0.5	& 2.57016731832103e-4	 &	0.081 &	116,93	& 6e-14	& 34\\
5	& 	1.28549969289457e-6 &		0.081 &	138,93 	& 9e-14	& 44\\
100	&	1.30803394394121e-9	&	0.081 &	166,93	 & -5e-14 &	37\\
250	&	1.58930738319053e-10	&0.081 &	174,93	 & -5e-15 &	39\\
\noalign{\smallskip}\hline
\end{tabular}
\end{table}
\begin{table}
\caption{CPDF and errors of $ConicC(1,1, 2)$ with the mesh
recommended for the error tolerance e-15.
$S_0$-parameters: $\al=1.3$, $\be=0.25$, $\sigma=0.001$, $\mu=0$.
Benchmark values are calculated using $ConicC(1,1, 2)$ with  the mesh longer and finer than recommended.
T: CPU time in microsec., the average over 10k runs. Mesh size $\ze$ is rounded.}
\label{tab:9}       % Give a unique label
\begin{tabular}{lllllll}
\hline\noalign{\smallskip}
$x$ & $F$ & $\ze$ & $N_-,N_+$ & Err & T \\
\noalign{\smallskip}\hline\noalign{\smallskip}
-250	&	1.83438084722098e-8	&	0.05 &	460,204 &	-e-15 	& 51 \\
-100	& 6.03684435773744e-8	&	0.05 &	441,204 &	-e-15& 51\\
-5	& 	2.96555322687464e-6	&	0.05 &	381,204	& -e-15 &	48\\
-0.5	&	5.91273879323451e-5	&	0.053 &	361,149 &	-e-12 &	42\\
-0.1	&  4.78178901456405e-4 &		0.053 &	331,149 &	 e-15	& 43 \\
0	&	0.475780098542004	&	0.049&	245,160 &	 e-15	& 39 \\
0.1	&	0.999195614410308 &		0.097 &	180,77 &	e-15 &	33\\
5	&	0.999995056257044	&	0.092 &	208,107 &	e-15&	38\\
100	&	0.999999899384689	&	0.092 &	241,107 &	e-15 &	37\\
250	&	0.999999969426835	&	0.092 &	251,107 &	e-15 &	38\\
\noalign{\smallskip}\hline
\end{tabular}
\end{table}

\subsection{Calculation of quantiles in the tails}  In Tables \ref{tab:10} -\ref{tab:12}, we calculate quantiles of the completely asymmetric stable L\'evy
process with the parameters $\al=0.7, \sg=0.001, \be=-1$, and $\al=0.15, \sg=0.001, \be=-1$. In the case $\al=0.7$, the rate of the tail decay is not exceedingly low,
and we calculate the quantiles $x_a$ in the range $[-10^{4},-200]$, for $a\in [10^{-5}, 10^{-4}]$. In the case $\al=0.15$, the tail decays extremely slow. We show the results
for $a\in [0.23, 0.56]$ and $a\in [0.105, 0.15]$. In the latter case, $x_a\in [-1500, -110]$; the results for larger regions further in the tail are similar: the same sets of conformal principal components
can be used for quantile calculations, hence, for Monte Carlosimulations over very large regions in the tail. The same is true for the case $a\in [0.23, 0.56]$, where the ratio
$x_{0.23}/x_{0.56}>2000$ is very large. In this particular case, the Newton method works but if the parameters of the process and $a$ are such that the Newton method is not applicable
(suppose, we do not know if $a$ is to the left of the inflection point), we can apply the bisection method. The results in the tables demonstrate that even if the initial guess for
the quantile is extremely rough ($x_a\in [-10000, -200]$ or $x_a\in [-1500, -110]$), the CPU time is only several times larger than in the case when the Newton method is applied.
In Tables \ref{tab:10}-\ref{tab:12}, we use the algorithm in Sect. \ref{algquantile} for stable processes of index $\al\in (0,1)$. 

\begin{table}
\caption{Quantiles $x_a$, errors w.r.t. the benchmark and  CPU time.
$S_0$-parameters: $\al=0.15$, $\be=-1$, $\sigma=0.001$, $\mu=0$.
In all cases, $ConicA(N_{sc}, 2)$ is used to calculate the conformal principal components
for the interval $[-10000, -200]$ where the quantiles are.
For the Newton method, the initial approximation is $x_0=-200$;
 for the bisection method, the initial interval is $[-10000, -200]$.
Benchmark: the Newton method is applied to $\ln F(x)-\ln a=0$; $N_{sc}=-2, \ze=0.165,	N_-=154, N_+=	19$.
A: the Newton method is applied to $\ln F(x)-\ln a=0$;  number
of iterations is in the range $[5, 8]$. B: the bisection method; number of iterations 44. In cases A,B, 
$N_{sc}=4, \ze=0.167, N_-=276, N_+=-69$. Time: CPU time in microsec., the average over 100k runs.}
\label{tab:10}       % Give a unique label
\begin{tabular}{lllllll}
\hline\noalign{\smallskip}
$a$ & $x_a$ & $Err_A$ & $T_A$ & $Err_B$ & $T_B$ \\
\noalign{\smallskip}\hline\noalign{\smallskip}
1e-5 &	-8973.08850717177 &			e-12	& 	37 &	e-12	&	181 \\
2e-5 &	-3333.5455711492 &			e-11 &		36 &			e-11 &	176\\
3e-5 &	-1867.90468266833	& 	e-12 &		 31 &		e-12	&	173\\
 5e-5 &	-900.414225337066	 &	e-12&	32 &			-e-12	&	171\\
7e-5	& -556.803989377748 &		e-12	& 27 &			e-12	&	169 \\
e-4 & -334.530078488661 & e-12 & 22 & -e12 &167 \\  
\noalign{\smallskip}\hline
\end{tabular}
\end{table}

\begin{table}
\caption{Quantiles $x_a$, errors w.r.t. the benchmark and  CPU time.
$S_0$-parameters: $\al=0.15$, $\be=-1$, $\sigma=0.001$, $\mu=0$.
In all cases, $ConicA(N_{sc}, 2)$ is used to calculate the conformal principal components
for the interval $[-1500, -100]$ where the quantiles are.
For the Newton method, the initial approximation is $x_0=-70$;
 for the bisection method, the initial interval is $[-1500, -70]$.
Benchmark: the Newton method is applied to $\ln F(x)-\ln a=0$; $N_{sc}=4, \ze=0.145,	N_-=498, N_+=	-71$.
A: the Newton method is applied to $\ln F(x)-\ln a=0$;  number
of iterations is in the range $[5, 8]$. B: the bisection method; number of iterations 41. In cases A,B, 
$N_{sc}=1, \ze=0.159,	N_-=313, N_+=	-18$. Time: CPU time in microsec., the average over 100k runs.}
\label{tab:11}       % Give a unique label
\begin{tabular}{lllllll}
\hline\noalign{\smallskip}
$a$ & $x_a$ & $Err_A$ & $T_A$ & $Err_B$ & $T_B$ \\
\noalign{\smallskip}\hline\noalign{\smallskip}
0.105 &	-1400.22243921946 &		2.8e-7	& 	72 &			7.9e-8	& 338 \\
0.115 &	-737.220889689652 &			1.4e-7	& 	62 &			3.8e-8	& 340\\
0.125 &	-408.160088631267 &			6.9e-8	&  60 &			2.0e-8 &	331\\
0.135 &	-235.779703690701 &			3.7e-8 &		51 &			1.1e-8 &	333\\
 0.15	 & -110.643637607915 &			1.6e-8 &		42 &			4.5e-9 &	329 \\  
\noalign{\smallskip}\hline
\end{tabular}
\end{table}

\begin{table}
\caption{Quantiles $x_a$, errors w.r.t. the benchmark and  CPU time.
$S_0$-parameters: $\al=0.15$, $\be=-1$, $\sigma=0.001$, $\mu=0$.
In all cases, $ConicA(N_{sc}, 2)$ is used to calculate the conformal principal components
for the interval $[-5, -0.002]$ where the quantiles are.
For the Newton method, the initial approximation is $x_0=-0.002$;
 for the bisection method, the initial interval is $[-5, -0.002]$.
Benchmark: the Newton method is applied to $\ln F(x)-\ln a=0$; $N_{sc}=10$, $\ze=0.128$,	$N_-=632$, $N_+=	-112$.
A: the Newton method is applied to $\ln F(x)-\ln a=0$; $N_{sc}=8, \ze=0.155, N_-=388, N_+=-6$, number
of iterations is in the range $[4, 10]$. B: the bisection method; $N_{sc}=8, \ze=0.155, N_-=388, N_+=-6$,
number of iterations 33. Time: CPU time in microsec., the average over 100k runs.}
\label{tab:12}       % Give a unique label
\begin{tabular}{lllllll}
\hline\noalign{\smallskip}
$a$ & $x_a$ & $Err_A$ & $T_A$ & $Err_B$ & $T_B$ \\
\noalign{\smallskip}\hline\noalign{\smallskip}
0.23 &	-4.72813632353329	& 		6.1e-9 &		92 &		6.2e-9 &	292\\
0.26	& -1.85093751685119&			2.1e-9 &	81 &		2.2e-9 &	278\\
0.29	& -0.789000640538996 &				8.3e-10 &	82 &	8.5e-10 &	285\\
0.35	& -0.173015534351966 &			1.6e-10	&  72 &		1.6e-10 &	276\\
0.44 &	-0.0241736559178538	&		1.9e-11 &	62 &		2.0e-11 &	280\\
0.5	& -7.30329034715694e-3&		5.6e-12	& 44 &		5.7e-12 &	276\\
0.53	 &-4.06193959209065e-3	& 	3.1e-12 &		44 &		3.1e-12 &	276\\
0.56	&-2.24338436545957e-3 &			-7e-12 &	35 &		-7e-12	& 275 \\  
\noalign{\smallskip}\hline
\end{tabular}
\end{table}

\section{Conclusion}\label{concl}
In the paper, we suggested three families of conformal deformations of the contours of integration
and the corresponding changes of variables in the Fourier
representations of the pdf and cpdf of stable L\'evy distributions and their derivatives w.r.t. $x$ and the
parameters of a stable L\'evy distribution, and developed numerical schemes for efficient numerical evaluation of the resulting integrals. 
 An appropriate change of variables $\xi=\xi(y)$ having being made, we apply the simplified trapezoid rule in the $y$-coordinate
\[
I=\ze\sum_{j=-N_-}^{N_+}f(y_j),\]
where $y_j=i\om+j\ze$, and $\om\in\bR$.
The first family increases exponentially the rate of decay of the integrand at infinity
(exponential acceleration), the second one leads to a polynomial increase of the rate of decay (polynomial acceleration),
and the third one increases the rate of decay but slower than polynomially (sub-polynomial acceleration).
Each family is divided into two: the first sub-family can be applied to completely asymmetric stable L\'evy distributions and other L\'evy distributions
whose characteristic functions admit analytic continuation to a union of a strip  around the real axis or adjacent to it and a cone,
while the second sub-family is applicable to general stable L\'evy distributions. 
For sub-families of the first kind, the rate of decay of the integrand in the new variable $y$ is the same as $y\to\pm\infty$
(we call them two-sided versions of the corresponding families), and one can take $N_-=N_+$. 
For sub-families of the second kind, the rate of decay as $\Re y\to-\infty$ is smaller than the rate of decay as $\Re y\to+\infty$,
hence, typically, $N_- \gg N_+$. Fortunately, one can derive asymptotic expansions for the truncated part of the infinite trapezoid rule at $-\infty$,
 calculate the corresponding sums, and significantly decrease $N_-$. In the case of the one-sided exponential acceleration,
 these sums are easy to calculate, and in the cases of the polynomial and sub-polynomial accelerations, the leading terms are expressible in terms
 of Riemann zeta function and its derivatives, which can be precalculated.
 
The exponential and polynomial accelerations are possible if the initial integrand admits analytic continuation to a cone, and decays
as $\xi\to \infty$ remaining in the cone; sub-polynomial acceleration works in some important cases where such a cone does not exist.
Clearly, if all types of acceleration are applicable, then, asymptotically, exponential acceleration is superior to polynomial acceleration, and the latter is superior to sub-polynomial one. However, for a fixed error tolerance (even very small, e.g., e-15), the mesh size $\ze$ can make the number of terms
in the simplified trapezoid rule much larger if exponential acceleration is applied. This happens if 
the ``width of the cone" (the length  of the intersection of the cone with the unit circle) is very small; sometimes, 
the polynomial and sub-polynomial accelerations
allow one to use  $\ze$ hundreds times larger than required by the exponential acceleration.

We described the regions in the parameter space where the most efficient exponential acceleration is preferable, and where
the other types are. In our numerical experiments, we observed that the polynomial and sub-polynomial accelerations 
work well where the exponential acceleration
is very inefficient. 
If $\al\in (0,1)$, then, for wide regions in the $(\sigma,\be,x')$- space,
very simple formulas can be used to calculate $p, p', F$ and quantiles.
For processes of index $\al\in (1,2)$, appropriate changes of variables can be used to reduce calculations to
integrals similar to the ones in the case $\al\in(0.5,1)$.  In cases when exponential acceleration is inefficient,
similar changes of variables reduce calculations in the case $\al\in (1/3,1)$ to the case $\al\in (1,3)$. Of course,
the case $\al>2$ is unrelated to stable distributions but the integrals that define the pdf and cpdf make sense for $\al>2$ as well.

%The calculations in the paper
%were performed in MATLAB 2017b-academic use, on a MacPro with a
%2.8 GHz Intel Core i7 and 16 GB 2133 MHz LPDDR3 RAM; 
For a wide region in the parameter space (about 90 percent of the total),
summation of 70-300 terms gives the pdf with the absolute error of order $10^{-15}-10^{-18}$,
and the cpdf with the absolute error of order $10^{-12}-10^{-15}$;  the CPU time is in the range 0.005-0.04 msec. Only for a relatively small region in the parameter space
(index $\al\neq 1$ is very close to 1, $x'/\sigma$ is not large, and $x'\be\tan(\al\pi/2)<0$; if $\al=1$, the condition is $x'\be<0$)
the number of terms is measured in thousands and more. However, the relative area in the parameter space where more than 1 msec
is needed is less than several percent. 

The methods of the paper can be applied to evaluation of various highly oscillatory integrals and special functions, as well as to
accurate and fast evaluation of Wiener-Hopf factors,  hence, to pricing barrier options, lookbacks, American options,
and other general first passage problems. In many cases, our methods are faster and more accurate than the saddle point method and methods based on the reduction to an appropriate cut in the complex plane. 
Contrary to Gaussian quadrature
schemes and their generalizations, no precalculation of nodes and weights with sufficiently high precision is required. On the contrary,
the schemes of our methods are flexible and simple, and two realizations can be used to check the accuracy of calculations; the standard repetitions
used in adaptive quadratures are unnecessary.
For calculation in tails of a stable distribution, 100-300 precalculated values of several simple expressions in the formula for the characteristic exponent at points of an appropriately chosen
grid suffice to evaluate pdf, cpdf and quantiles in wide regions not too close to 0, very far in the tails including. We suggest to call these expressions {\em conformal principal components}.
In many cases, the CPU time is less than 0.1 msec.;
 hence, in applications
to the Monte Carlosimulations, it becomes essentially unnecessary to truncate the state space. For fat-tailed distributions,
the truncation is a serious source of errors. More involved conformal changes of variables can be designed to decrease
the number of terms and the CPU time; we leave the study of this possibility for the future.

\end{document}